\documentclass[12pt,leqno,oneside]{amsart}
\usepackage{mathrsfs,dsfont}
\usepackage{amsmath,amstext,amsthm,amssymb,amscd,bbm}
\usepackage{charter}
\usepackage{typearea}
\usepackage{pdfsync}
\usepackage{color}
\usepackage{mathtools}
\usepackage{enumitem}
\usepackage{typearea}

\usepackage[backref=page]{hyperref}

\usepackage[width=6.4in,height=8.5in]
{geometry}

\numberwithin{equation}{section}

\newtheorem{Theorem}{Theorem}[section]
\newtheorem{Proposition}[Theorem]{Proposition}

\newtheorem{Corollary}[Theorem]{Corollary}
\theoremstyle{definition}
\newtheorem{Definition}[Theorem]{Definition}

\newtheorem{Remark}[Theorem]{Remark}

\newcommand{\ov}{\overline}

\DeclareMathOperator{\codim}{codim}

\DeclareMathOperator{\Span}{Span}

\DeclareMathOperator{\rank}{rank}

\DeclareMathOperator{\End}{End}
\DeclareMathOperator{\Id}{Id}
\newcommand{\cali}[1]{\mathscr{#1}}
\newcommand{\cO}{\cali{O}}

\newcommand{\cT}{\cali{T}}
\newcommand{\cC}{\cali{C}}

\newcommand{\field}[1]{\mathbb{#1}}
\newcommand{\Z}{\field{Z}}
\newcommand{\R}{\field{R}}
\newcommand{\C}{\field{C}}
\newcommand{\N}{\field{N}}

\renewcommand{\P}{\field{P}}

\newcommand{\E}{\field{E}}
\newcommand{\G}{\field{G}}

\newcommand{\FS}{\mathrm{FS}}

\newcommand{\bs}{{\boldsymbol{s}}}

\newcommand{\comment}[1]{}

\begin{document}

\title{Asymptotic expansion of induced Grassmannian Chern forms and 
distribution of random degeneracy sets}

\author{Turgay Bayraktar} 
\thanks{T. Bayraktar is partially supported by T\"{U}B\.{I}TAK grant ARDEB-1001/124F370}
\address{Faculty of Engineering and Natural Sciences, 
Sabanc{\i} University, \.{I}stanbul, Turkey}
\email{tbayraktar@sabanciuniv.edu}
\urladdr{https://orcid.org/0000-0002-1364-9728}

\author{Dan Coman}
\thanks{D.\ Coman is supported by the NSF Grant DMS-2154273}
\address{Department of Mathematics, Syracuse University, 
Syracuse, NY 13244-1150, USA}
\email{dcoman@syr.edu}
\urladdr{https://orcid.org/0009-0003-9311-9464}

\author{Bingxiao Liu}
\address{School of Sciences, Great Bay University, Dongguan 523000, China} 
\email{liubx@gbu.edu.cn}
\thanks{B.\ Liu is supported by the University Start-up Fund 
(Grant No. YJKY260008) from Great Bay University and partially supported by DFG Priority Program 2265 
`Random Geometric Systems' (Project-ID 422743078) from 
Universit\"at zu K\"oln.}
\urladdr{https://orcid.org/0000-0001-7124-0413}

\author{George Marinescu}
\address{Universit\"at zu K\"oln, Mathematisches institut,
Weyertal 86-90, 50931 K\"oln, Germany 
\newline\mbox{\quad}\,Institute of Mathematics `Simion Stoilow', 
Romanian Academy, Bucharest, Romania}
\email{gmarines@math.uni-koeln.de}
\thanks{G.\ Marinescu is partially supported 
by the DFG funded projects SFB TRR 191 `Symplectic Structures in Geometry, 
Algebra and Dynamics' (Project-ID 281071066\,--\,TRR 191),
DFG Priority Program 2265 `Random Geometric Systems' 
(Project-ID 422743078), the ANR-DFG project `QuasiDy\,--\,Quantization, Singularities, 
and Holomorphic Dynamics' (Project-ID 490843120)}
\urladdr{https://orcid.org/0000-0001-6539-7860}

\subjclass[2020]{Primary 32L10; Secondary 32A60, 
32U40, 53C55, 60D05.} %81Q50.}

\keywords{Bergman kernel, Tian's theorem, Grassmannian embedding,
Chern forms, holomorphic vector bundle, random holomorphic section, degeneracy
locus, meromorphic transform}

\date{June 20, 2026}

\dedicatory{Dedicated to Professor Yum-Tong Siu with admiration and gratitude}

\begin{abstract}
For the Grassmannian embeddings defined by the spaces \(H^0(X,L^p\otimes E)\), where
\(L\) is a positive line bundle and \(E\) is a holomorphic vector bundle over a compact
complex manifold, we prove a complete asymptotic expansion of the induced
Grassmannian Chern forms and compute the first coefficients explicitly.
As an application of the first-order asymptotics and of the theory of meromorphic 
transforms by Dinh and
Sibony, we prove that on a compact K\"ahler manifold, 
the normalized currents of integration over the loci where
several random sections become linearly dependent converge almost surely to the
corresponding power of the curvature form of the positive line bundle, with a
quantitative estimate for the speed of convergence. Moreover, in the determinant case, we additionally present an alternative method based on the Wishart distribution, together with variance estimates.
\end{abstract}

\maketitle

\tableofcontents

\section{Introduction}\label{S:Intro}
The study of zero sets of random holomorphic sections lies at the intersection of
complex geometry, probability, and pluripotential theory. Following the foundational
work of Shiffman and Zelditch \cite{ShZ99}, many authors have investigated the
equidistribution and fluctuation properties of such zeros in geometric settings; see,
for instance, \cite{BCM,Sh08,DMS12,DMM16,DMN17,DLM25}. This line of research is
also closely related to the classical theory of zeros of random polynomials, where
analogous distribution questions have been studied for a long time. We refer the reader
to \cite{BL13,BCHM} and the references therein for further background.
In the geometric setting, one studies holomorphic sections of high tensor
powers of a positive line bundle and asks how their zero currents behave in the
semiclassical limit. For line bundles, the answer is governed by the curvature form of
the underlying Hermitian metric. In \cite{BCLM}, we extended this circle of ideas to
sections of vector bundles of the form $L^p\otimes E$, using Grassmannian embeddings
and a higher rank version of Tian's theorem. 

The present paper has two closely related goals. First, we strengthen the
Grassmannian Tian approximation theorem from \cite{BCLM} 
by proving a complete asymptotic expansion
for the pullback of Chern forms of the dual universal bundle. Second, we use this
deterministic expansion to prove an almost sure equidistribution theorem not only for zero sets
of individual random sections but also for degeneracy loci of random $k$-tuples of sections.

Let $(X,\omega)$ be a (connected) compact Hermitian manifold of dimension $n$, 
where $\omega$ is a Hermitian metric compatible with the complex structure $J$ of $X$. 
Set $g^{TX}(\cdot,\cdot):=\omega(\cdot, J\,\cdot)$ the associated Riemannian metric 
with the Riemannian volume form $dv_X=\frac{\omega^n}{n!}$.

For any given vector bundle $F\to X$, together with a smooth Hermitian metric 
$h^F$ and connection $\nabla^F$, we introduce the following notation: 
a sequence of sections $B_p\in \mathscr{C}^\infty(X,F)$, $p\in\N$ 
admits an asymptotic expansion on $X$ as $p\to \infty$
if there exists an integer $M\in\Z$ and a sequence of smooth sections 
$\mathcal{B}_m\in \mathscr{C}^\infty(X,F)$ ($m\in\N$) 
such that for any $\ell\in \N$ and $N\in\N$, there exists 
a constant $C_{\ell, N}>0$ such that for all $p$
\begin{align}\label{e:expansion}
&\Big \|B_p- \sum_{m=0}^{N}\mathcal{B}_m
p^{M-m} \Big \|_{\cC^\ell(X)} 
\leqslant C_{\ell,N}\, p^{M-N-1},
\end{align}
where the $\cC^\ell(X)$-norm is induced by the Levi-Civita connection 
of $(X,g^{TX})$ and $h^F$, $\nabla^F$. In this case, we simply write
$$B_p\sim \sum_{m=0}^\infty p^{M-m} \mathcal{B}_m.$$

We begin by recalling the geometric setting from \cite{BCLM} on which
the present paper builds. Let $(L,h^L)$ be a positive line bundle on $X$, and let $(E,h^E)$ be a
Hermitian holomorphic vector bundle of rank $r\leq n$ on $X$. For $p\geq1$, we set
\[
L^p:=L^{\otimes p},\qquad V_p:=H^0(X,L^p\otimes E),\qquad d_p:=\dim V_p-1.
\]
We let $h^{L^p}$ and $h^{L^p\otimes E}$ be the Hermitian metrics induced by $h^L$ and
$h^E$ on $L^p$, respectively, on $L^p\otimes E$. We endow $V_p$ with the $L^2$-inner
product induced by this metric data, namely
\begin{equation}\label{e:ip1}
(S,S')_p=\int_X\langle S,S'\rangle_{h^{L^p\otimes E}}\,\frac{\omega^n}{n!}\,,\qquad
S,S'\in V_p .
\end{equation}
We endow its dual space $V_p^\star=H^{0}(X,L^{p}\otimes E)^\star$ with the inner product
induced by the one on $V_p$.

Let $\G(r,V_p^\star)$ be the Grassmannian of complex $r$-dimensional subspaces of
$V_p^\star$. The Kodaira map is defined by
\begin{equation}\label{e:Kod1}
\Phi_p:X\to\G(d_p+1-r,V_p)=\G(r,V_p^\star)\,,\qquad
\Phi_p(x)=\{s\in V_p:\,s(x)=0\}.
\end{equation}
Since $(L,h^L)$ is positive, the above map is a well-defined holomorphic embedding for all
$p$ sufficiently large; see \cite[Theorem 5.1.18]{MM07}.
Moreover, $c_1(L,h^L)$ defines a K\"{a}hler metric on $X$. 
Let $b_0\in\mathscr{C}^\infty(X,\R_{>0})$ and 
$\alpha_L\in\Omega^{(1,1)}(X,\R)$ be defined by
\begin{equation}
    b_0(x):=\frac{c_1(L,h^L)^n}{\omega^n}(x) >0,\qquad \alpha_L:=\frac{\sqrt{-1}}{2\pi}\partial\ov{\partial}\log b_0 \in\Omega^{(1,1)}(X,\R).
    \label{e:bzero}
\end{equation}
Let $r^X_L$ denote the scalar curvature of the K\"{a}hler manifold $(X, c_1(L,h^L))$.

Let $\cT$ be the universal holomorphic vector bundle over $\G(r,V_p^\star)$, and let
$h^\cT$ be the Hermitian metric on it induced by the inner product on $V_p^\star$. Let
$(\cT^\star,h^{\cT^\star})$ be the dual holomorphic vector bundle. 
We denote by $c(\cT^\star,h^{\cT^\star})$ 
and $c_k(\cT^\star,h^{\cT^\star})$, respectively, the total Chern form 
and the $k$-th Chern form of
$(\cT^\star,h^{\cT^\star})$ (see Definition \ref{def:Chern}). 
The next theorem strengthens the Grassmannian version of 
Tian’s approximation theorem established in \cite[Theorem 1.1]{BCLM}. 
In particular, we derive a complete asymptotic expansion for 
the pullback of Chern forms and explicitly compute the first three terms in each degree.
\begin{Theorem}\label{T:Tian}
Let $(X,\omega)$ be a compact Hermitian manifold of dimension $n$. Let
$(L,h^L)$ be a positive line bundle on $X$, and let $(E,h^E)$ be 
a Hermitian holomorphic vector bundle on $X$ of rank $r\leq n$. 

(1) For sufficiently large $p$, we have the formula for the total Chern form
\begin{equation}
    \Phi_p^\star(c(\cT^\star, h^{\cT^\star}))=
    \sum_{k=0}^r \mathfrak{S}_{r-k}\left(\frac{\sqrt{-1}}{2\pi}R^E+
    \frac{\sqrt{-1}}{2\pi}\widehat{A}_p\right)\left(pc_1(L,h^L)+
    \alpha_L+1\right)^k,
    \label{e:Chern}
\end{equation}
where $\mathfrak{S}_{r-k}(\cdot)$ denotes the elementary invariant polynomial 
of degree $r-k$ (see \eqref{e:symmpol}), and we use the notation in 
Proposition \ref{p:ApExp} for $\widehat{A}_p:=
A_p-\mathcal{A}_0\sim \sum_{m=1}^\infty p^{-m}\mathcal{A}_m$.

(2) For each $0\leq k\leq r$ and $\ell=0,1,2,\ldots,$
there exists a form $\boldsymbol{T}_{k,\ell}\in \Omega^{(k,k)}(X,\R)$, 
which can be expressed as a polynomial in $R^{TX}$, $R^L$, $R^E$, $d\omega$, 
and their covariant derivatives, such that, as $p\to\infty$, 
we have the following asymptotic expansion,
\begin{equation}\label{e:Tian}
\Phi_p^\star(c_k(\cT^\star,h^{\cT^\star})) 
\sim \sum_{\ell=0}^\infty p^{k-\ell} \boldsymbol{T}_{k,\ell}.
\end{equation}
Moreover,  we have
\begin{equation}
\begin{split}
\boldsymbol{T}_{k,0}=&\binom rk c_1(L,h^L)^k\,,
\qquad\boldsymbol{T}_{k,1}=\binom{r-1}{k-1}c_1(L,h^L)^{k-1}
\wedge \left(c_1(E,h^E)+r\alpha_L \right),\\
 \boldsymbol{T}_{k,2}=& \binom{r-2}{k-2} c_1(L,h^L)^{k-2}
 \wedge\left(c_2(E,h^E)+(r-1)c_1(E,h^E)\wedge\alpha_L+
 \frac{r(r-1)}{2}\alpha_L^2\right)\\
 &\,+\binom{r-1}{k-1}c_1(L,h^L)^{k-1}\wedge
 \left(r\frac{\sqrt{-1}}{16\pi^2}\partial\ov{\partial} r^X_L + 
 \frac{\sqrt{-1}}{2\pi}\partial\ov{\partial}\left(\Lambda_{\omega_L}c_1(E,h^E)+
 r\Lambda_{\omega_L}\alpha_L\right)\right),
\end{split}
\label{e:Tian1}
\end{equation}
where $\sqrt{-1}\Lambda_{\omega_L}c_1(E,h^E), 
\sqrt{-1}\Lambda_{\omega_L}\alpha_L\in\mathscr{C}^\infty(X,\R)$ 
are defined as in \eqref{e:ContrL} and \eqref{e:DeltaL}, 
and our convention is that $c_1(L,h^L)^0=1$ and for an integer 
$m<0$, $c_1(L,h^L)^m:=0$.
\end{Theorem}
%===
In \cite[Theorem 1.1]{BCLM}, we obtained the initial coefficients 
of the asymptotic expansion, namely $\boldsymbol{T}_{k,0}$, 
and when $\omega=c_1(L,h^L)$, also $\boldsymbol{T}_{k,1}$. 
These computations provide the first instances of the more general 
identity \eqref{e:Chern} and the asymptotic expansion \eqref{e:Tian} 
developed here.
For the proof of Theorem \ref{T:distrib}, only the leading term in
Theorem \ref{T:Tian} is needed. The full expansion, however, provides 
a more precise
Tian-type approximation theorem for Grassmannian embeddings and may 
be useful in related problems involving higher-order statistics 
of random degeneracy currents.

Note that Tian's original argument \cite{Ti90} relies on H\"ormander's 
$L^2$ estimates to produce peak sections. 
The $L^2$ estimates used in Tian's construction of peak sections are naturally 
intertwined with singular weights arising from plurisubharmonic potentials and 
closed positive currents, a point of view already central in Siu's work on 
Lelong numbers and positive currents \cite{Siu74} and in Skoda's $L^2$ 
techniques for weighted ideals of holomorphic functions \cite{Skoda72}, 
and later developed systematically by Demailly \cite{Demailly92}
and Nadel \cite{Nadel90}.
%through Demailly's regularization theory 
%\cite{Demailly92} and Nadel's multiplier ideal sheaves 
%and vanishing theorem \cite{Nadel90}.
In the context of singular metrics, Tian's theorem was considered 
in relation to the equidistribution of zeros in
\cite{CM15,DMM16,CMN16,CMM17,CMN18,CMN24}.

We now turn from the deterministic asymptotics of Grassmannian embeddings to their
probabilistic consequences for random sections.
In \cite{BCLM}, the Grassmannian Tian approximation theorem was used to study the
distribution of zero currents of random sections in \(V_p\). 
By Bertini's theorem, for a generic section $s\in V_p$, the
zero set $Z_s=\{s=0\}$ is a complex submanifold of $X$ with codimension $r$. We denote by
$[s=0]$ the current of integration over $Z_s$. We endow the projective space $\P V_p$
with the probability measure
\[
\Upsilon_p=\omega_\FS^{d_p},
\]
where $\omega_\FS$ denotes the Fubini–Study form on $\P V_p$, and we consider the
product probability space
\[
(\mathcal H,\Upsilon)=
\left(
\prod_{p=1}^\infty \P V_p,\,
\prod_{p=1}^\infty \Upsilon_p
\right).
\]
It was shown in \cite[Theorem 1.2]{BCLM} that for $\Upsilon$-almost every sequence
$\{s_p\}_{p\geq1}\in\mathcal H$,
\[
p^{-r}[s_p=0]\longrightarrow c_1(L,h^L)^r
\]
weakly as currents as $p\to\infty$. Moreover, the convergence holds with speed
$O(p^{-1}\log p)$ outside exceptional sets of probability $O(p^{-\gamma})$, for every
$\gamma>1$.

The present paper proves the corresponding result for higher degeneracy loci. For
$1\leq k\leq r$, let
\[
(\P V_p)^k=\P V_p\times\cdots\times\P V_p .
\]
If $\bs_p=(s^p_1,\ldots,s^p_k)\in(\P V_p)^k$, we define the degeneracy set
\begin{equation}\label{e:degset-p}
D_k(\bs_p)=D_k(s^p_1,\ldots,s^p_k)
=
\{x\in X:\,s^p_1(x)\wedge\cdots\wedge s^p_k(x)=0\}.
\end{equation}
Thus $D_k(\bs_p)$ is the set of points $x\in X$ where
$s^p_1(x),\ldots,s^p_k(x)$ are linearly dependent; see \cite[p.\ 411]{GH94}. For
generic $(s^p_1,\ldots,s^p_r)\in(\P V_p)^r$, the degeneracy loci have the expected
dimensions: for $2\leq k\leq r$, the set $D_k(s^p_1,\ldots,s^p_k)$ has codimension
$r+1-k$ and is smooth away from $D_{k-1}(s^p_1,\ldots,s^p_{k-1})$, while
$D_1(s^p_1)=Z_{s^p_1}$ is a complex submanifold of $X$ with codimension $r$; see
\cite[pp.\ 411--412]{GH94}. Hence, the current of integration $[D_k(\bs_p)]$ along
$D_k(\bs_p)$ is a well-defined positive closed current of bidegree $(r+1-k,r+1-k)$ on
$X$.

We endow the multiprojective space $(\P V_p)^k$ with the product 
probability measure
\[
\Upsilon_{p,k}:=\Upsilon_p^k,
\]
induced by the Fubini--Study volume $\Upsilon_p=\omega_\FS^{d_p}$ on each factor. We
then consider the product probability space
\begin{equation}\label{e:prob}
(\mathcal H_k,\Upsilon_k)=
\left(
\prod_{p=1}^\infty(\P V_p)^k,\,
\prod_{p=1}^\infty\Upsilon_{p,k}
\right).
\end{equation}

%June 21st 2026
In \cite[Theorem 1.4 and Corollary 1.5]{BCLM} (see also Theorem \ref{T:expdeg}), letting $\bs_p$ denote the random element in $(\P V_p)^k$ distributed according to the probability measure $\Upsilon_{p,k}$, we established that, for all sufficiently large integers $p$, the expectation current of $[D_k(\bs_p)]$ exists and is given by
\[
\E\big[[D_k(\bs_p)]\big]=\Phi_p^\star\big(c_{r+1-k}(\cT^\star,h^{\cT^\star})\big).
\]
Combining this with Theorem \ref{T:Tian}, we refine \cite[Corollary 1.5]{BCLM} to the following asymptotic expansion:
\begin{equation}
    \frac{1}{p^{r+1-k}}\E\big[[D_k(\bs_p)]\big]
= \binom{r}{k-1} c_1(L,h^L)^{r+1-k} 
+\frac{1}{p}\boldsymbol{T}_{r+1-k,1}
+\frac{1}{p^2}\boldsymbol{T}_{r+1-k,2}
+O(p^{-3}),
\label{e:Easym}
\end{equation}
where $\boldsymbol{T}_{r+1-k,1}$ encodes the information corresponding to $c_1(E,h^E)$, and $\boldsymbol{T}_{r+1-k,2}$ encodes the information corresponding to $c_1(E,h^E)$ and $c_2(E,h^E)$.
%%%%%%%

Parallel to the limiting expectation given in \eqref{e:Easym}, our second main theorem describes the almost sure asymptotic distribution of 
the currents of integration along these random degeneracy sets.

\begin{Theorem}\label{T:distrib}
Let $(X,\omega)$ be a compact K\"ahler manifold of dimension $n$, let $(L,h^L)$ be a
positive line bundle on $X$, and let $(E,h^E)$ be a Hermitian holomorphic vector bundle
of rank $r\leq n$ on $X$. Then there exist $C>0$ and $p_0\in\N$ such that the following
holds.

For any $1\leq k\leq r$ and $\gamma>1$, there exist subsets
$\mathcal{E}_{p,k}=\mathcal{E}_{p,k}(\gamma)\subset(\P V_p)^k$ such that for $p>p_0$,
we have

\emph{(i)} $\Upsilon_{p,k}(\mathcal{E}_{p,k})\leq Cp^{-\gamma}$;

\emph{(ii)} if $\bs_p\in(\P V_p)^k\setminus \mathcal{E}_{p,k}$, then
\begin{equation}\label{e:espeed}
\left|
\left\langle
\frac{1}{p^{r+1-k}}\,[D_k(\bs_p)]
-
\binom{r}{k-1} c_1(L,h^L)^{r+1-k},
\phi
\right\rangle
\right|
\leq
C\gamma\,\frac{\log p}{p}\,\|\phi\|_{\cC^2(X)},
\end{equation}
for any $(n+k-r-1,n+k-r-1)$-form $\phi$ of class $\cC^2$ on $X$. Moreover, the
estimate \eqref{e:espeed} holds for $\Upsilon_k$-almost every sequence
$\{\bs_p\}_{p\geq1}\in\mathcal H_k$, provided that $p$ is large enough. Hence
\[
\frac{1}{p^{r+1-k}}\,[D_k(\bs_p)]
\longrightarrow
\binom{r}{k-1} c_1(L,h^L)^{r+1-k}
\]
in the weak sense of currents as $p\to\infty$, almost surely.
\end{Theorem}

For $k=1$, Theorem \ref{T:distrib} recovers the equidistribution theorem 
for zeros of random sections from \cite[Theorem 1.2]{BCLM}.  
The new content is the treatment of all higher degeneracy currents. 
The normalization $p^{-(r+1-k)}$ and the constant $\binom{r}{k-1}$ 
are dictated by the fact that the Poincar\'e dual of a generic degeneracy locus 
of $k$ sections of a rank $r$ vector bundle is the Chern class $c_{r+1-k}$.

Theorem \ref{T:distrib} gives a quantitative almost sure equidistribution statement.
It strengthens the corresponding expectation statement by showing that the limiting
behavior holds for individual random degeneracy currents, outside exceptional sets of
explicitly controlled probability. The parameter $\gamma$ measures the strength of
this probabilistic exclusion. Taking $\gamma$ large makes the exceptional sets
$\mathcal E_{p,k}(\gamma)$ very small, since their mass is bounded by
$Cp^{-\gamma}$, but this comes at the cost of a larger constant in the error estimate.
Conversely, smaller values of $\gamma$ give a sharper quantitative estimate for the
convergence speed while allowing larger exceptional sets. 
%Thus, the theorem exhibits a
%balance between probability and accuracy: one may require the estimate to hold outside
%rarer exceptional events, at the price of weakening the uniform numerical bound on the
%convergence rate. 
Since the exceptional probabilities are summable for every
$\gamma>1$, the Borel–Cantelli lemma yields the almost sure convergence asserted in
the theorem.

The paper is organized as follows. In Section \ref{S:Tian}, we prove the asymptotic
expansion of the induced Grassmannian Chern forms given in Theorem \ref{T:Tian}. 
In Section \ref{S:MTD}, we recall and develop the meromorphic transforms 
\cite{DS06} associated with
degeneracy loci of tuples of holomorphic sections. 
%The proof uses the method of meromorphic transforms developed by Dinh and Sibony
%\cite{DS06}. 
For $k=1$, the relevant transform is the one associated in \cite{BCLM}
with the Kodaira map. For general $k$, we introduce a meromorphic transform associated
with the incidence relation
\(
x\in D_k(s_1,\ldots,s_k).
\)
More precisely, in section \ref{S:MTD}, we study the analytic set
\(\Gamma_{E,k}=\{(x,s)\in X\times(\P V)^k:\,x\in D_k(s)\},\)
where $V=H^0(X,E)$, and show that it defines a meromorphic transform from $X$ to
$(\P V)^k$. We compute its relevant intermediate degrees in terms of Chern classes of
$E$. In section \ref{S:distrib}, we apply the Dinh–Sibony equidistribution theorem to
the sequence of transforms associated with $L^p\otimes E$. 
Together with the leading term of Theorem
\ref{T:Tian}, this proves the almost sure equidistribution theorem for random
degeneracy currents. Finally, in Section \ref{S:Wick}, we give a complementary local
argument in the determinant case $k=r$ based on the Wishart distribution, together with variance estimates.
%Finally, Theorem \ref{T:Tian} identifies the deterministic average current with its semiclassical limit,
%which yields Theorem \ref{T:distrib}.

\medskip
\emph{
We dedicate this work to Professor  Yum-Tong Siu, whose foundational vision of using $L^2$-methods, 
positivity, and holomorphic sections to bridge complex differential geometry and 
algebraic geometry forms part of the conceptual background against which 
Tian's approximation theorem and its higher-rank analogs naturally stand.
His mathematical legacy continues to shape complex analysis, geometry, and algebraic geometry, 
and to inspire future generations of mathematicians. 
}

%%%%%
%\section{Tian's approximation theorem for Grassmannian embeddings}\label{S:Tian}
\section{Grassmannian Tian asymptotics}\label{S:Tian}

In this section we prove Theorem \ref{T:Tian}. We consider the Kodaira map $\Phi_p$
to the Grassmannian \eqref{e:Kod1}
%\[
%\Phi_p:X\longrightarrow G(r,V_p^\ast),
%\qquad V_p=H^0(X,L^p\otimes E),
%\]
and study the Chern forms on \(X\) obtained by pulling back the Chern forms of the dual
tautological bundle $(\cT^\star,h^{\cT^\star})$. The key observation is that, under the
natural identification \(\Phi_p^\ast(\cT^\ast)\simeq L^p\otimes E\), the induced metric is
described by the diagonal Bergman kernel endomorphism \(P_p(x)\). This gives a curvature
formula for \(\Phi_p^\ast(\cT^\ast)\), and the diagonal Bergman kernel expansion then
implies the full asymptotic expansion of the induced Grassmannian Chern forms. The first
coefficients are computed explicitly at the end of the section.

Let $(X,\omega)$ be a (connected) compact Hermitian manifold of dimension $n$, 
where $\omega$ is a Hermitian metric compatible with the complex structure $J$ of $X$. 
Set $g^{TX}(\cdot,\cdot):=\omega(\cdot, J\,\cdot)$ the associated Riemannian metric 
with the Riemannian volume form $dv_X=\frac{\omega^n}{n!}$.

Let $(L,h^L)$
be a positive line bundle on $X$, and let $(E,h^E)$ be a Hermitian holomorphic vector
bundle of rank $r\leq n$ on $X$. For $p\geq1$, we set 
\[L^p:=L^{\otimes p},\;V_p:=H^0(X,L^p\otimes E),\;d_p:=\dim V_p-1.\]
We let $h^{L^p}$ and $h^{L^p\otimes E}$ be the Hermitian metrics induced by 
$h^L$ and $h^E$ on $L^p$, respectively, on $L^p\otimes E$. Let $\nabla^E$ and $\nabla^L$ denote the Chern connections of $(E,h^E)$ and $(L,R^L)$, respectively, with Chern curvatures $R^E$ and $R^L$.

We equip the space $V_p$ with the $L^2$-inner product defined in \eqref{e:ip1}, and we endow its dual space $V_p^\star = H^{0}(X, L^{p}\otimes E)^\star$ with the inner product canonically induced by that on $V_p$. Furthermore, we denote by $\mathcal{L}^2(X, L^{p}\otimes E)$ the $L^2$-space associated with the $L^2$ inner product on $X$ determined by the Riemannian metric $g^{TX}$ and the Hermitian metrics $h^{L^{p}}$ and $h^{E}$.
Denote by $H^{0}(X, L^{p}\otimes E)$ the vector space of holomorphic sections of the bundle $L^{p}\otimes E$ over $X$. This space is a finite-dimensional complex vector subspace of $\mathcal{L}^{2}(X, L^{p}\otimes E)$.
Let 
\begin{equation}
P_p:\mathcal{L}^2(X, L^{p}\otimes E)\to
H^{0}(X, L^{p}\otimes E)
	\label{eq:3.1.4}
\end{equation}
be the orthogonal (Bergman) projection.

The Schwartz kernel $P_p(\cdot,\cdot)$ of $P_p$
with respect to the volume form $dv_X$ is called the Bergman kernel of $H^{0}(X, L^{p}\otimes E)$.
It is a smooth section of $(L^p\otimes E)\boxtimes
(L^p\otimes E)^\star$ over $X\times X$, 
\begin{equation}\label{bk2.2}
P_p(x,x') \in (L^p\otimes E)_x\otimes (L^p\otimes E)_{x'}^\star\,,
\end{equation}
especially, 
\begin{equation}\label{bk2.3}
P_p(x,x) \in \End(L^p\otimes E)_x = \End E_x,
\end{equation}
where we use the canonical identification $\End(L^p)=\C$ 
for any line bundle $L$ on $X$.
Let $\{S^p_j\}_{j=0}^{d_p}$, $d_p := 
\dim H^{0}(X,L^p\otimes E)-1$ 
be any orthonormal basis of
$H^{0}(X,L^p\otimes E)$ with respect to the $L^2$ inner
product.
Then we have
% \begin{equation} \label{bk2.4}
% P_p(x,x')=\sum_{j=0}^{d_p} S^p_j (x) \otimes S^p_j(x')^*
% \in (L^p\otimes E)_x\otimes (L^p\otimes E)_{x'}^*.
% \end{equation}
% and 
\begin{equation} \label{bk2.5}
P_p(x):=P_p(x,x)= \sum_{j=0}^{d_p} S^p_j (x) \otimes S^p_j(x)^*
\in \End(E_x).
\end{equation}

Set $\omega_L:=c_1(L,h^L)$, by the positivity assumption, $\omega_L$ is a K\"{a}hler metric on $X$. Let $g^{TX}_L$ denote the associated Riemannian metric. 
Let $r^X_L$, $\Delta_L$ be the scalar curvature and let denote the (positive) Bochner Laplacian associated to $g^{TX}_L$. Let $\Lambda_{\omega_L}R^E\in\mathscr{C}^\infty(X,\mathrm{End}(E))$ denote the contraction of $R^E$ with respect to $\omega_L$ and $g^{TX}_L$. If $\{w_j\}_{j=1}^n$ is an orthonormal basis of $(T^{(1,0)}X, g^{TX}_L)$, then 
\begin{equation}
\sqrt{-1}\Lambda_{\omega_L}R^E=\sum_j R^E(w_j,\ov{w}_j).
    \label{e:ContrL}
\end{equation}
Note that for any smooth function $f$ on $X$, we have
\begin{equation}
    \Lambda_{\omega_L}(\sqrt{-1}\partial\ov{\partial} f)=-\frac{1}{2}\Delta_L f.
    \label{e:DeltaL}
\end{equation}

We have the following diagonal asymptotic expansion
of the Bergman kernel. 
%===
\begin{Theorem}[{\cite[Theorems 4.1.1 and 4.1.3]{MM07}}]\label{bkt2.1} 
For every $m\in\N$, there exists a smooth section
\(\boldsymbol{b}_m\in \mathscr{C}^{\infty}(X,\End E),\)
such that 
\begin{align}\label{bk2.6}
P_p(x)\sim \sum_{m=0}^{\infty} p^{n-m}\,\boldsymbol{b}_m(x).
\end{align}
The value $\boldsymbol{b}_m(x)$ is a universal polynomial in the
curvature tensors $R^{TX}$ and $R^E$, their covariant derivatives
of order at most $2m-2$, and the tensors $d\omega$ and $R^L$, 
together with their covariant derivatives of orders at most
$2m-1$ and $2m$, respectively. These expressions also involve
reciprocals of linear combinations of the eigenvalues of
$\sqrt{-1}R^L_x$, computed with respect to the metric $g^{TX}$.
Moreover, we have
\begin{align}\label{abk2.5}
&\boldsymbol{b}_0=b_0 \Id_{E},\\
&\boldsymbol{b}_1=\frac{b_0}{8\pi} (r^X_L- 2\Delta_L \log b_0)\Id_{E}+\frac{b_0}{2\pi}\sqrt{-1}\Lambda_{\omega_L}R^E.
\end{align}

%Furthermore, the expansion is uniform in the following manner.
%For any fixed integers $k,\ell\in \N$, assume that the derivatives of 
%$g^{TX}$, $h^L$, and $h^E$ of order $\leqslant 2n+2k+\ell+6$ are 
%continuous over a set bounded in the $\cC^\ell$-\,norm taken 
%with respect to the parameter $x\in X$. 
%Additionally, assume that $g^{TX}$ is continuous over a set bounded below. 
%Then the constant $C_{k,\,\ell}$ is independent of $g^{TX}$ and 
%the $\cC^\ell$-norm in \eqref{bk2.6} encompasses also the 
%derivatives with respect to the parameters. 
\end{Theorem}
\begin{Remark}\label{r:2.2}
When the bundle $(L,h^L)$ polarizes $(X,\omega)$,
that is $\omega=\omega_L=c_1(L,h)$, we have $b_0\equiv 1$. Hence
    \begin{align}\label{abk2.5-2}
&\boldsymbol{b}_0=\Id_{E},\\
&\boldsymbol{b}_1=\frac{ r^X_L}{8\pi}\Id_{E}+\frac{1}{2\pi}\sqrt{-1}\Lambda_{\omega_L}R^E.
\end{align}
Moreover, in this case, an explicit formula of $\boldsymbol{b}_2$ was given in \cite[eq. (0.8)]{MMCrelle} (see also \cite[Theorem 2.16]{MM11}).
\end{Remark}

% Let $\G(r,V_p^\star)$ be the Grassmannian of $r$-dimensional subspaces of $V_p^\star$. The Kodaira 
% map is defined by
% \begin{equation}\label{e:Kod1}
% \Phi_{p}:X\to\G(d_p+1-r,V_p)=\G(r,V_p^\star)\,,\,\;\Phi_{p}(x)=\{s\in V_p:\,s(x)=0\}.
% \end{equation}
% Since $L$ is positive, the above map is a well-defined 
% holomorphic embedding for all $p$ sufficiently large 
% (see \cite[Theorem 5.1.18]{MM07}).

Let $\cT$ be the universal holomorphic vector bundle over $\G(r,V_p^\star)$ 
and $h^\cT$ be the Hermitian metric on it induced by the inner product on $V_p^\star$. 
Let $(\cT^\star,h^{\cT^\star})$ be the dual holomorphic vector bundle.

If $s\in V_p$, the evaluation map $\varepsilon_s:V_p^\star\to\C$, 
$\varepsilon_s(f)=f(s)$ defines a holomorphic 
section $\sigma_{s}\in H^0(\G(r,V_p^\star),\cT^\star)\equiv V_p\,$, 
where $\sigma_s(W)=\varepsilon_s|_W$ 
and $W\in\G(r,V_p^\star)$.
For $p$ large enough, let $\Phi_{p}:X\to\G(d_p+1-r,V_p)=\G(r,V_p^\star)$ be the Grassmannian embedding defined in \eqref{e:Kod1}. Then we have the following isomorphism of 
holomorphic vector bundles on $X$ (see \cite[Theorem 5.1.16]{MM07}),
\begin{equation}
	\begin{split}
&\Psi_{p}:\Phi_{p}^\star(\cT^\star)\rightarrow \;L^{p}\otimes E,\\
&\Psi_{p}((\Phi_{p}^\star\sigma_{s})(x))=s(x),\;\text{ for any $s\in V_p$, $x\in X$.}
	\end{split}
	\label{e:isostar}
\end{equation}
Moreover, under this isomorphism, we have
\begin{equation}
	h^{\Phi_{p}^\star(\cT^\star)}(x)=h^{L^{p}\otimes E}(x)\circ 
	P_{p}(x)^{-1},
	\label{e:isostarP}
\end{equation}
where $h^{\Phi_{p}^\star(\cT^\star)}$ is the Hermitian metric on 
$\Phi_{p}^\star(\cT^\star)$ induced by $h^{\cT^\star}$.

Let \(\nabla^{E(1,0)}\) and \(\nabla^{E(0,1)}\) denote, respectively, the \((1,0)\)- and \((0,1)\)-components of the Chern connection \(\nabla^{E}\).
Let $\nabla^{\End(E)}=[\nabla^E,\cdot]$ be the connection on $\End(E)$ induced by $\nabla^E$, and let $\nabla^{\End(E),(1,0)}$ and $\nabla^{\End(E),(0,1)}$ denote the $(1,0)$- and $(0,1)$-components of $\nabla^{\End(E)}$. For all sufficiently large values of \(p\), the smooth section \(P_p(x) \in \mathscr{C}^\infty\bigl(X,\mathrm{End}(E)\bigr)\) is pointwise invertible on \(X\). We therefore define its pointwise inverse for every $x\in X$ by
\[
Q_p(x) := P_p(x)^{-1} \in \mathrm{End}(E).
\]
\begin{Proposition}
    Assume that $p$ is sufficiently large. Let $h^{\Phi_{p}^\star(\cT^\star)}$ denote the Hermitian metric on the bundle $L^p \otimes E \to X$ induced by the identification $\Psi_{p}$ in \eqref{e:isostar}. Then the Chern connection 
    of the Hermitian holomorphic vector bundle $(L^p \otimes E, h^{\Phi_{p}^\star(\cT^\star)})$ is given by the following expression:
    \begin{equation}\label{e:C}
    \Phi_{p}^\star\left(\nabla^{\cT^\star}\right)=\nabla^{L^p\otimes E}+P_p(x)\nabla^{\End(E),(1,0)}Q_p(x).
     \end{equation}
 Consequently, its curvature is given by the expression      \begin{equation}\label{e:Rstar}
       \Phi_{p}^\star\left(R^{\cT^\star}\right)=R^{L^p\otimes E}+\nabla^{\End(E),(0,1)}\left(P_p(x)\nabla^{\End(E),(1,0)}Q_p(x)\right)\in\Omega^{(1,1)}(X,\mathrm{End}(E)).
   \end{equation}
\end{Proposition}
\begin{proof}
    By \eqref{e:isostarP} and using the Chern connection $\nabla^{L^p\otimes E}$ on $(L^p\otimes E,h^{L^p\otimes E})$, for any two local sections $s_1, s_2$ of $L^p\otimes E$, we have
\begin{equation}\label{e:C1}
 \begin{split}
dh^{\Phi_{p}^\star(\cT^\star)}(s_1,s_2)
= &\, h^{\Phi_{p}^\star(\cT^\star)}(P_p(x)[\nabla^{L^p\otimes E},Q_p(x)]s_1, s_2)
\\
&+h^{\Phi_{p}^\star(\cT^\star)}(\nabla^{L^p\otimes E}s_1, s_2)+h^{\Phi_{p}^\star(\cT^\star)}(s_1, \nabla^{L^p\otimes E}s_2).
 \end{split}
 \end{equation}

 Note that $P_p(x)$ and $Q_p(x)$ are Hermitian endomorphisms of $(E_x, h^{L^p\otimes E}_x)$, consequently
 \begin{equation}\label{e:C2}
 h^{\Phi_{p}^\star(\cT^\star)}(P_p(x)[\nabla^{L^p\otimes E},Q_p(x)]s_1, s_2)=h^{\Phi_{p}^\star(\cT^\star)}(s_1, P_p(x)[\nabla^{L^p\otimes E},Q_p(x)]s_2).
 \end{equation}

Combining \eqref{e:C1} and \eqref{e:C2}, we conclude that $\nabla^{L^p\otimes E}+P_p(x)[\nabla^{E(1,0)},Q_p(x)]$ is a Hermitian metric of $(L^p\otimes E, h^{\Phi_{p}^\star(\cT^\star)})$, it is exactly the connection in \eqref{e:C}. Moreover, its $(0,1)$-component is the Dolbeault operator $\ov{\partial}^{L^p\otimes E}$; hence, the connection in \eqref{e:C} defines the unique Chern connection associated with $h^{\Phi_{p}^\star(\cT^\star)}$. The formula \eqref{e:Rstar} follows from \eqref{e:C}.
\end{proof}

\begin{Proposition}\label{p:ApExp}
Define $A_p\in\Omega^{(1,1)}(X,\mathrm{End}(E))$ by
    \begin{equation}\label{e:Ap}
    \begin{split}
          A_p:=&\nabla^{\End(E),(0,1)}\left(P_p(x)\nabla^{\End(E),(1,0)}Q_p(x)\right)\\
          =&\left[\nabla^{E(0,1)},P_p(x)[\nabla^{E(1,0)},Q_p(x)]\right]\in\Omega^{(1,1)}(X,\mathrm{End}(E)).
    \end{split}
    \end{equation}
 Then, as $p\to\infty$, $A_p$ admits an asymptotic expansion 
    \begin{equation}\label{e:ApExp}
   A_p\sim\sum_{m=0}^\infty p^{-m} \mathcal{A}_m,
    \end{equation}
    where $ \mathcal{A}_m\in\Omega^{(1,1)}(X,\mathrm{End}(E))$ and
    \begin{equation}
\begin{split}
    \mathcal{A}_{0}=& \partial\ov{\partial}\log b_0\, \Id_E,\\
 \mathcal{A}_{1}=& -\nabla^{\End(E),(0,1)}(\nabla^{\End(E),(1,0)} (b_0^{-1}\boldsymbol{b}_1))\\
 =&\frac{1}{8\pi}
\partial\bar\partial
\left(r_L^X-2\Delta_L\log b_0\right)\Id_E
-\frac{\sqrt{-1}}{2\pi}\nabla^{\End(E),(0,1)}\nabla^{\End(E),(1,0)}
\left(\Lambda_{\omega_L}R^E\right),\\
\mathcal{A}_2=&-\frac{1}{b_0}\mathcal{A}_1 \boldsymbol{b}_1-\nabla^{\End(E),(1,0)}(b_0^{-1}\boldsymbol{b}_1)\wedge\nabla^{\End(E),(0,1)}(b_0^{-1}\boldsymbol{b}_1)\\
&\quad - \nabla^{\End(E),(0,1)}(\nabla^{\End(E),(1,0)} (b_0^{-1}\boldsymbol{b}_2)),
\end{split}
\label{e:ApTian1}
\end{equation}
where the function $b_0$ is defined in \eqref{e:bzero} and the sections $\boldsymbol{b}_1$, $\boldsymbol{b}_2$ are given in the expansion \eqref{bk2.6} of $P_p(x)$.
\end{Proposition}
\begin{proof}
We assume $p$ to be sufficiently large, then by \eqref{bk2.6} and the fact $b_0 >0$, we have the asymptotic expansion of $Q_p\in \mathscr{C}^\infty(X,\mathrm{End}(E))$, write
\begin{align}\label{e:QpExp}
Q_p(x)\sim \sum_{m=0}^{\infty} p^{-n-m}  \boldsymbol{q}_m(x),
\end{align}
where
    \begin{equation}
\begin{split}
   &  \boldsymbol{q}_{0}= \frac{1}{b_0} \Id_E
;\\
&  \boldsymbol{q}_{1}= -\frac{1}{b_0^2}\boldsymbol{b}_1\\
& \boldsymbol{q}_2=\frac{1}{b_0^3}\left(\boldsymbol{b}_1^2-b_0\boldsymbol{b}_2\right)
\end{split}
\label{e:QpT}
\end{equation}

Using the expansions \eqref{bk2.6} with \eqref{e:QpExp} in the definition \eqref{e:Ap} of $A_p$, we obtain the existence of the asymptotic expansion \eqref{e:ApExp}. Now we focus on the calculations of $\mathcal{A}_{m}$ for $m=0,1,2$. In fact, we have by elementary computations
 \begin{equation}
\begin{split}
 P_p(x)[\nabla^{E(1,0)},Q_p(x)]   =&  b_0\partial b_0^{-1}\, \Id_E + p^{-1}\left(b_0[\nabla^{E(1,0)},\boldsymbol{q}_1]+\partial b_0^{-1} \boldsymbol{b}_1\right)\\
& +p^{-2}\left(b_0[\nabla^{E(1,0)},\boldsymbol{q}_2]+\partial b_0^{-1} \boldsymbol{b}_2+\boldsymbol{b}_1[\nabla^{E(1,0)},\boldsymbol{q}_1]\right)+O(p^{-3}).
\end{split}
\label{e:Ap1}
\end{equation}
As a consequence, we obtain
$\mathcal{A}_{0}=[\nabla^{E(0,1)}, b_0\partial b_0^{-1}\Id_E]=\partial\ov{\partial}\log b_0\, \Id_E$ and
    \begin{equation}
\begin{split}
&  \mathcal{A}_{1}= \left[\nabla^{E(0,1)}, b_0[\nabla^{E(1,0)},\boldsymbol{q}_1]+\partial b_0^{-1} \boldsymbol{b}_1\right],\\
& \mathcal{A}_2=\left[\nabla^{E(0,1)},b_0[\nabla^{E(1,0)},\boldsymbol{q}_2]+\partial b_0^{-1} \boldsymbol{b}_2+\boldsymbol{b}_1[\nabla^{E(1,0)},\boldsymbol{q}_1]\right].
\end{split}
\label{e:Ap2}
\end{equation}
At last, combining \eqref{e:QpT} and \eqref{e:Ap2}, we obtain the formulas in \eqref{e:ApTian1}.
\end{proof}

\begin{Remark}
In the course of establishing \cite[Theorem 1.1]{BCLM}, the form
\(A_p\) also arises (specifically, in \cite[eq.~(3.11)]{BCLM}) via local computations performed in holomorphic local frames of the bundles \(L\) and \(E\). Furthermore, a direct verification shows that the definition of \(A_p\) given in \eqref{e:Ap} agrees with the corresponding form appearing in \cite{BCLM}.
%
     % In the proof of \cite[Theorem 1.1]{BCLM}, we also obtain a term $A_p$ (more precisely, in \cite[eq. (3.11)]{BCLM}) via the local computations using local holomorphic frames of $L$ and $E$. In fact, one can verify directly that $A_p$ in \eqref{e:Ap} is exactly the same term.
\end{Remark}

\begin{Remark}\label{r:2.6}
   By Remark \ref{r:2.2}, when the bundle $(L,h^L)$ polarizes $(X,\omega)$,that is $\omega=\omega_L=c_1(L,h)$, we have explicit formulas for all $\boldsymbol{b}_0,\boldsymbol{b}_1, \boldsymbol{b}_2$, so that we can also work out more explicit formulas for $\mathcal{A}_2$. In particular, we also have
   $$\mathcal{A}_0\equiv 0,\, \mathcal{A}_1=\frac{1}{8\pi}
\partial\bar\partial r_L^X \Id_E
-\frac{\sqrt{-1}}{2\pi}\nabla^{\End(E),(0,1)}\nabla^{\End(E),(1,0)}
\left(\Lambda_{\omega_L}R^E\right).$$
%The scalar curvature is the contraction of the Ricci form of the dual canonical bundle ........
\end{Remark}

The following result gives an extension of classical Tian's approximation theorem \cite{Ti90, Z98, Ca99} to the vector bundle case.
\begin{Corollary}
    As $p\to\infty$, the curvature form $\Phi_{p}^\star\left(R^{\cT^\star}\right)\in\Omega^{(1,1)}(X,\End(E))$ admits an asymptotic expansion 
    \begin{equation}\label{e:RpExp}
   \Phi_{p}^\star\left(R^{\cT^\star}\right)\sim\sum_{m=0}^\infty p^{1-m} \mathcal{R}_m,
    \end{equation}
    where $ \mathcal{R}_m\in\Omega^{(1,1)}(X,\mathrm{End}(E))$ and
    \begin{equation}
\begin{split}
    \mathcal{R}_{0}=& R^L \Id_E,\\
        \mathcal{R}_{1}=& R^E + \partial\ov{\partial}\log b_0\, \Id_E,\\
 \mathcal{R}_{2}=&\frac{1}{8\pi}
\partial\bar\partial
\left(r_L^X-2\Delta_L\log b_0\right)\Id_E
-\frac{\sqrt{-1}}{2\pi}\nabla^{\End(E),(0,1)}\nabla^{\End(E),(1,0)}
\left(\Lambda_{\omega_L}R^E\right)\\
=& \frac{\partial\bar\partial r_L^X}{8\pi}
\Id_E
-\frac{\sqrt{-1}}{2\pi}\nabla^{\End(E),(0,1)}\nabla^{\End(E),(1,0)}
\left(\Lambda_{\omega_L}\mathcal{R}_{1}\right), \\
\mathcal{R}_3=&-\frac{1}{b_0}\mathcal{R}_2 \boldsymbol{b}_1-\nabla^{\End(E),(1,0)}(b_0^{-1}\boldsymbol{b}_1)\wedge\nabla^{\End(E),(0,1)}(b_0^{-1}\boldsymbol{b}_1)\\
&\quad - \nabla^{\End(E),(0,1)}(\nabla^{\End(E),(1,0)} (b_0^{-1}\boldsymbol{b}_2)),
\end{split}
\label{e:RpTian1}
\end{equation}
where the function $b_0$ is defined in \eqref{e:bzero} and the sections $\boldsymbol{b}_1$, $\boldsymbol{b}_2$ are given in the expansion \eqref{bk2.6} of $P_p(x)$. 
\end{Corollary}
\begin{proof}
    The expansion \eqref{e:RpExp} follows from \eqref{e:Rstar} and \eqref{e:ApExp}. The formulas in \eqref{e:RpTian1} follow from \eqref{e:DeltaL} and \eqref{e:ApTian1}.
\end{proof}

We also have the following consequence.
\begin{Corollary}
  If  the curvature form $\Phi_{p}^\star\left(R^{\cT^\star}\right)\in\Omega^{(1,1)}(X,\End(E))$ admits the following asymptotic expansion as $p\to\infty$, 
    \begin{equation}\label{e:RpCor}
   \Phi_{p}^\star\left(R^{\cT^\star}\right)=pR^L \Id_E+O(p^{-2}),
    \end{equation}
    that is, $\mathcal{R}_1=\mathcal{R}_2\equiv 0 \in\Omega^{(1,1)}(X,\mathrm{End}(E))$ in \eqref{e:RpExp}, then $(E,b_0^{-1}h^E)$ is unitarily flat and $(X,\omega_L)$ has a constant scalar curvature. The converse also holds.
\end{Corollary}

Let $A$ be an $r\times r$ matrix of complex numbers. Recall that the {\em elementary invariant polynomials} $\mathfrak{S}_k$ are defined by (see, e.g., \cite[p.\ 402]{GH94}) 
\[\det(A+t\cdot\mathrm{Id}_r)=\sum_{k=0}^r\mathfrak{S}_{r-k}(A)t^k,\;t\in\C,\]
where $\mathrm{Id}_r$ is the $r\times r$ identity matrix. 
We have that $\mathfrak{S}_k(U^{-1}AU)=\mathfrak{S}_k(A)$ for any invertible $r\times r$ matrix $U$ and 
\begin{equation}\label{e:symmpol}
\mathfrak{S}_k(A)=\sum_{\sharp I=k}\det A_{I,I},
\end{equation}
where $A_{I,J}$ denotes the $(I,J)$-th minor $(A_{ij})_{i\in I,j\in J}$ of $A$. Since the wedge product is commutative on differential forms of degree $2$, $\mathfrak{S}_k(A)$
is well-defined in the same way for any $r\times r$ matrix $A$ of forms of degree $2$, namely:
\begin{equation}\label{e:det}
\det(A+t\cdot\mathrm{Id}_r)=\sum_{k=0}^r\mathfrak{S}_{r-k}(A)\wedge t^k.
\end{equation}
where $t$ is a form of degree $2$ or a number.
Then 
\begin{equation}\label{e:Chern-m}
c_k(A):=\mathfrak{S}_k(A)
\end{equation} 
is a form of degree $2k$ given by \eqref{e:symmpol}.

\begin{Definition}\label{def:Chern}
The total Chern form associated to a Hermitian holomorphic vector 
bundle $(E,h^{E})$ of rank $r$ is defined by
\begin{equation}\label{e:Chern1}
c(E,h^{E})=\det\Big(\frac{\sqrt{-1}}{2\pi}R^{E}+\mathrm{Id}_{E}\Big)\in
\bigoplus_{k=0}^{r} \Omega^{(k,k)}(X,\R).
\end{equation}
The $(k,k)$-component of $c(E,h^{E})$ is called the $k$-th Chern form of $(E,h^E)$, that is,
\begin{equation}\label{e:ckform}
c_k(E,h^E):=\mathfrak{S}_k\left(\frac{\sqrt{-1}}{2\pi}R^{E}\right)\in \Omega^{(k,k)}(X,\R)
\end{equation} 
\end{Definition}

\begin{proof}[Proof of Theorem \ref{T:Tian}]
By \eqref{e:Rstar}, \eqref{e:RpExp} and the definition of the total Chern form, we obtain
\begin{equation}\label{e:ChernPf}
\begin{split}
    \Phi_p^\star(c(\cT^\star, h^{\cT^\star}))&=\det\left(\frac{\sqrt{-1}}{2\pi}\Phi_p^\star(R^{\cT^\star})+\mathrm{Id}_{E}\right)\\
    &=\det\left(\frac{\sqrt{-1}}{2\pi}\left(R^E+\widehat{A}_p\right)+(pc_1(L,h^L)+\alpha_L+1)\mathrm{Id}_{E}\right).
\end{split}
\end{equation}
Then \eqref{e:Chern} follows directly from \eqref{e:det}.

Since $\mathfrak{S}_k$ is a polynomial in its variable and $\widehat{A}_p$ admits an asymptotic expansion, we obtain the existence of the asymptotic expansion of $\Phi_p^\star(c(\cT^\star, h^{\cT^\star}))$ on $X$ as $p\to \infty$. Then, after taking the $(k,k)$-components, we obtain the expansion \eqref{e:Tian}.

Now we need to calculate $\boldsymbol{T}_{k,\ell}$, $\ell=0,1,2$. Taking the $(k,k)$-components of \eqref{e:Chern}, we have
\begin{equation}\label{e:ckid}
\Phi_p^\star(c_k(\cT^\star,h^{\cT^\star}))=\sum_{\ell=0}^k\binom{r-k+\ell}{\ell} (pc_1(L,h^L)+\alpha_L)^\ell \mathfrak{S}_{k-\ell}\left(\frac{\sqrt{-1}}{2\pi}\left(R^E+\widehat{A}_p\right)\right).
\end{equation}

Note that by Proposition \ref{p:ApExp}, we have
\begin{equation}
    \widehat{A}_p=\frac{1}{8\pi p}
\partial\bar\partial
\left(r_L^X-2\Delta_L\log b_0\right)\Id_E
-\frac{\sqrt{-1}}{2\pi p}\nabla^{\End(E),(0,1)}\nabla^{\End(E),(1,0)}
\left(\Lambda_{\omega_L}R^E\right) +O(p^{-2}),
\end{equation}
Then $\mathfrak{S}_{0}\left(\frac{\sqrt{-1}}{2\pi}\left(R^E+\widehat{A}_p\right)\right)\equiv 1$ and
\begin{equation}
\begin{split}
     \mathfrak{S}_{1}\left(\frac{\sqrt{-1}}{2\pi}\left(R^E+\widehat{A}_p\right)\right)= &  c_1(E,h^E) +r\frac{\sqrt{-1}}{16\pi^2 p}\partial\ov{\partial} r^X_L \\
     & + \frac{\sqrt{-1}}{2\pi p}\partial\ov{\partial}\left(\Lambda_{\omega_L}c_1(E,h^E)+r\Lambda_{\omega_L}\alpha_L\right)+O(p^{-2}),\\
      \mathfrak{S}_{k}\left(\frac{\sqrt{-1}}{2\pi}\left(R^E+\widehat{A}_p\right)\right)= &c_k(E,h^E)+O(p^{-1}), \, 0\leq k\leq r.
\end{split}
\label{e:proof1}
\end{equation}

For the term $\boldsymbol{T}_{k,0}$, it corresponding to the term $\ell=k$ in \eqref{e:ckid}, we obtain the first formula in \eqref{e:Tian1}.

When $B_p$ has an asymptotic expansion $B_p\sim \sum_m p^{m}\mathcal{B}_m$, set $(B_p)^{[m]}:=\mathcal{B}_m$. Then we have by \eqref{e:ckid},
\begin{equation}
\begin{split}
\boldsymbol{T}_{k,1} =& \binom{r}{k}\left((pc_1(L,h^L)+\alpha_L)^{k}\right)^{[k-1]}\\
&+\binom{r-1}{k-1}\left((pc_1(L,h^L)+\alpha_L)^{k-1}\right)^{[k-1]}\mathfrak{S}_{1}\left(\frac{\sqrt{-1}}{2\pi}\left(R^E+\widehat{A}_p\right)\right)^{[0]}\\
\boldsymbol{T}_{k,2} =& \binom{r}{k}\left((pc_1(L,h^L)+\alpha_L)^{k}\right)^{[k-2]}\\
&+\binom{r-1}{k-1}\left((pc_1(L,h^L)+\alpha_L)^{k-1}\right)^{[k-2]}\mathfrak{S}_{1}\left(\frac{\sqrt{-1}}{2\pi}\left(R^E+\widehat{A}_p\right)\right)^{[0]}\\
&+\binom{r-1}{k-1}\left((pc_1(L,h^L)+\alpha_L)^{k-1}\right)^{[k-1]}\mathfrak{S}_{1}\left(\frac{\sqrt{-1}}{2\pi}\left(R^E+\widehat{A}_p\right)\right)^{[-1]}\\
&+\binom{r-2}{k-2}\left((pc_1(L,h^L)+\alpha_L)^{k-2}\right)^{[k-2]}\mathfrak{S}_{2}\left(\frac{\sqrt{-1}}{2\pi}\left(R^E+\widehat{A}_p\right)\right)^{[0]}
\end{split}
\label{e:proof2}
\end{equation}
At last, combining \eqref{e:proof1} and \eqref{e:proof2}, we obtain the rest part of \eqref{e:Tian1}. The proof is completed.
\end{proof}

\begin{Remark}
Since we already have an explicit expression for the polynomial $\mathfrak{S}_{2}$, 
we can obtain an explicit formula for $\boldsymbol{T}_{k,3}$ 
by applying the same procedure as in \eqref{e:proof2}.
If the function $b_0 = c_1(L, h^L)^n / \omega^n$ is constant on $X$ 
(for instance, when the bundle $(L, h^L)$ polarizes $(X, \omega)$, 
that is, when $c_1(L, h) = \omega$), then $\alpha_L$ vanishes identically. 
In this situation, the expressions for $\boldsymbol{T}_{k,1}$ and $\boldsymbol{T}_{k, 2}$ in \eqref{e:Tian1} simplify.
\end{Remark}

%%%%%

\section{Meromorphic transforms and degeneracy sets}\label{S:MTD}

%We study in this section certain meromorphic transforms related to 
%degeneracy sets of holomorphic 
%sections, which will be used for the proof of Theorem \ref{T:distrib}.

In this section, we introduce the geometric object that will be used later to pass from
expectation formulas to almost sure equidistribution statements. To a holomorphic vector
bundle \(E\) generated by its global sections and to an integer \(1\leq k\leq r=\rank E\), we
associate a meromorphic transform whose graph is the incidence relation
\(x\in D_k(s_1,\ldots,s_k)\).
This transform encodes the family of degeneracy loci of \(k\)-tuples of sections as the
fibers of a single analytic correspondence.

The main results of the section are Theorems \ref{T:MTD} and \ref{T:MTGdeg}. First, we
show that the incidence set defining the transform is irreducible, has the expected
dimension, and gives a meromorphic transform from \(X\) to \((\P V)^k\). Second, we
compute the relevant top intermediate degrees of this transform. These degrees are
expressed in terms of the Chern classes \(c_{r+1-k}(E)\) and \(c_{r-k}(E)\). This
identification is the key input needed in Section \ref{S:distrib}, where the
Dinh–Sibony equidistribution theorem is applied to the sequence of transforms
associated with the bundles \(L^p\otimes E\).

We will work in the following general setting:  

\medskip

(A) $(X,\omega)$ is a (connected) compact K\"ahler manifold of dimension $n$; $(E,h^E)$
is a Hermitian holomorphic vector bundle of rank $r$ on $X$, and
\[1\leq r\leq n,\;V:=H^0(X,E),\;N:=\dim V-1,\;\P V=\P H^0(X,E).\]

\smallskip

(B) For every section $S\in V$, $Z_S:=\{x\in X:\,S(x)=0\}\neq\varnothing$.

\smallskip

(C) $N\geq r$ and for every $x\in X$, $V$ spans the fiber $E_x$ of $E$ over $x$.

\medskip

We note in Section \ref{S:distrib} that assumptions (A)-(C) hold for 
vector bundles of the form $L^p\otimes E$, where $L$ is a 
positive line bundle and $p$ is sufficiently large. 
We endow $V$ with the $L^2$-inner product 
\begin{equation}\label{e:ip2}
(S,S')=\int_X\langle S,S'\rangle_{h^E}\,\frac{\omega^n}{n!}\,,\,\;S,S'\in V,
\end{equation}
defined as in \eqref{e:ip1} using the Hermitian metric $h^E$. 
We consider the dual $V^\star$ with the inner product determined
by that on $V$, and we let $\omega_\FS$ be the induced Fubini-Study
form on the $N$-dimensional projective space $\P V$. 

Let $\G(k,V)$ denote the Grassmannian of (complex) $k$-planes in $V$. 
By (C), the Kodaira map 
\begin{equation}\label{e:Kod2}
\Phi_E:X\to\G(N+1-r,V)=\G(r,V^\star),\;\Phi_E(x)=\{s\in V:\,s(x)=0\},
\end{equation}
is well defined and holomorphic. Let $\cT\to \mathbb{G}(r,V^\star)$ 
be the universal bundle endowed with the Hermitian metric $h^{\cT}$ 
induced by the inner product on $V^\star$, and let 
$(\cT^\star,h^{\cT^\star})$ be the dual holomorphic vector bundle.

Let $1\leq k\leq r$. Recall that, given $\bs=(s_1,\ldots,s_k)\in(\P V)^k$ (or $\bs=(s_1,\ldots,s_k)\in V^k$), 
the degeneracy set $D_k(\bs)$ is defined by
\begin{equation}\label{e:degset}
D_k(\bs)=D_k(s_1,\ldots,s_k)=\{x\in X:\,s_1(x)\wedge\ldots\wedge s_k(x)=0\}.
\end{equation}
For generic $\bs=(s_1,\ldots,s_k)\in(\P V)^k$, $D_k(\bs)$ is an analytic subset of $X$ 
of pure dimension $n+k-r-1$, smooth away from the degeneracy set $D_{k-1}(s_1,\ldots,s_{k-1})$. 
Moreover, Bertini's theorem for vector bundles \cite[Theorem 1.2]{MZ23} 
implies that $Z_s$ is a complex submanifold of $X$ of dimension $n-r$
for all $s$ outside a proper analytic subset of $\P V$. Hence the current of integration $[D_k(\bs)]$ 
along $D_k(\bs)$ is a well-defined positive closed current of bidegree $(r+1-k,r+1-k)$ on $X$. 

Let $\mu_k$ be a probability measure on $(\P V)^k$, and 
$\E_k(\mu_k)=\E([D_k(s_1,\ldots,s_k)],\mu_k)$ 
be the expectation current of the current valued random variable 
\[(\P V)^k\ni(s_1,\ldots,s_k)\longmapsto[D_k(s_1,\ldots,s_k)],\] 
defined by 
\begin{equation}\label{e:expk}
\langle \E_k(\mu_k),\phi \rangle=
\int_{(\P V)^k}\langle[D_k(s_1,\ldots,s_k)],\phi\rangle\,d\mu_k\,,
\end{equation}
where $\phi$ is a smooth $(n+k-r-1,n+k-r-1)$ form on $X$. 
Using the fact that the Poincar\'e dual
of $D_k(s_1,\ldots,s_k)$ is $c_{r+1-k}(E)$, we obtained in \cite{BCLM} 
a precise formula for the expectation currents in the case 
of Gaussian or Fubini-Study volumes.
%===
\begin{Theorem}[{\cite[Theorem 1.4]{BCLM}}]\label{T:expdeg}
Let $(X,\omega),\,(E,h^E)$ verify (A)-(C). 
Assume that the Kodaira map $\Phi_E:X\to\G(r,V^\star)$ 
defined in \eqref{e:Kod2} is an embedding. 
For $1\leq k\leq r$, let $\nu_k$ be the Gaussian probability measure 
on $V^k$ and $\mu_k$ be the product measure on $(\P V)^k$
determined by the Fubini-Study volume $\omega_\FS^N$ on $\P V$.
Then the expectation currents $\E_k(\mu_k), \E_k(\nu_k)$ are well defined positive closed currents 
and are given by 
\[\E_k(\mu_k)=\E_k(\nu_k)=
\Phi_E^\star\big(c_{r+1-k}(\cT^\star,h^{\cT^\star})\big),\;1\leq k\leq r.\]
\end{Theorem}

Here, the expectation current $\E_k(\nu_k)$ is defined as 
in \eqref{e:expk} for the probability space $(V^k,\nu_k)$. For $k=1$, 
Theorem \ref{T:expdeg} was established using different methods 
by J.\ Sun \cite[Theorem 4.2]{Sun}. 

\medskip

Let us recall the notion of meromorphic transform introduced by Dinh and Sibony \cite{DS06}. If $X_1,\,X_2$ are connected compact K\"ahler manifolds of dimension $k_1,\,k_2$ and $l$ is an integer such that $k_1-k_2\leq l\leq k_1$, a meromorphic transform $F$ from $X_1$ to $X_2$ is an irreducible analytic subset $\Gamma\subset X_1\times X_2$ of dimension $k_2+l$ such that the restrictions of the canonical projections $X_1\times X_2\to X_j$ to $\Gamma$ are surjective (see \cite[Section\ 3.1]{DS06} for the general definition). One calls $\Gamma$ the graph of $F$ and $l$ the codimension of $F$. 

\smallskip

We now introduce the meromorphic transforms associated with degeneracy sets. 
Let $(X,\omega),\,(E,h^E),\,V,\,N$ verify assumptions (A)-(C).
For $1\leq k\leq r$, let $F_{E,k}$ be the meromorphic transform from $X$ to $(\P V)^k$ with graph
\begin{equation}\label{e:Gamma_k}
\Gamma_{E,k}=\{(x,\bs)\in X\times(\P V)^k:\,x\in D_k(\bs)\}\subset X\times(\P V)^k.
\end{equation}
Let $\pi_{1,k}:X\times(\P V)^k\to X$, $\pi_{2,k}:X\times(\P V)^k\to(\P V)^k$ be the canonical projections.
Note that for $k=1$, $F_{E,1}=F_E$ is the meromorphic transform associated to the Kodaira map $\Phi_E$ defined in \eqref{e:Kod2} that was employed in \cite{BCLM} (see\cite[Proposition 4.2]{BCLM}).

\begin{Theorem}\label{T:MTD}
Let $(X,\omega),\,(E,h^E)$ verify assumptions (A)-(C). Then $\Gamma_{E,k}$ is an irreducible analytic subset of $X\times(\P V)^k$ of dimension $kN+n+k-r-1$ and $\Gamma_{E,k}\setminus(\Gamma_{E,k-1}\times\P V)$ is a connected complex manifold. Moreover, $F_{E,k}$ is a meromorphic transform of codimension $n+k-r-1$.
\end{Theorem}

\begin{proof} We show first that $\Gamma_{E,k}$ is an analytic subset of $X\times(\P V)^k$. 
If $S_0,\ldots,S_N$ is a basis of $V$, we use the identification  
$\P^N\ni[x_0:\ldots:x_N]\longmapsto x_0S_0+\ldots+x_NS_N\in\P V$.
Let $U\subset X$ be an open set such that $E|_U$ 
has a holomorphic frame $e_1,\ldots,e_r$, and write 
\[S_j=\sum_{\ell=1}^rh_{j\ell}e_\ell, \text{ where } 
h_{j\ell}\in\cO_X(U).\]
If $\bs=(s_1,\ldots,s_k)\in(\P V)^k$ there exist 
$a_m=[a_{m0}:\ldots:a_{mN}]\in\P^N$ such that 
\[s_m=\sum_{j=0}^Na_{mj}S_j=
\sum_{\ell=1}^r\left(\sum_{j=0}^Na_{mj}h_{j\ell}\right)e_\ell,\;
1\leq m\leq k.\]
Therefore $(x,\bs)\in\Gamma_{E,k}\cap(U\times(\P^N)^k)$ if and only if $\rank A(x,\bs)\leq k-1$, where 
\begin{equation}\label{e:defeqdeg}
A(x,\bs)=A(x,a_1,\ldots,a_k)=
\left[\sum_{j=0}^Na_{mj}h_{j\ell}(x)\right]_{1\leq m\leq k,1\leq\ell\leq r}
\end{equation}
is a $k\times r$ matrix of holomorphic functions on $(\P^N)^k\times U$. 
Thus $\Gamma_{E,k}\cap(U\times(\P^N)^k)$ is defined by $\binom{r}{k}$ equations, given by the 
vanishing of all the $k\times k$ minors of $A(x,\bs)$. 

We now proceed by induction on $k$. For $k=1$, it follows 
by assumption (C) that the linear map $S\in V\to S(x)\in E_x$ 
has rank $r$ for every $x\in X$, so $\{s\in\P V:\,s(x)=0\}$ 
is a linear subspace of dimension $N-r$ of $\P V$ and $\Gamma_E$ 
is a connected complex submanifold of $X\times\P V$ of dimension 
$N+n-r$ (a projective bundle over $X$).

Assume next that $\Gamma_{E,k-1}$ is an irreducible analytic subset of $X\times(\P V)^{k-1}$ of dimension $(k-1)N+n+k-r-2$. Then $\mathcal U=(X\times(\P V)^{k-1})\setminus\Gamma_{E,k-1}$ is a Zariski open, hence a connected, dense, open subset of $X\times(\P V)^{k-1}$. Writing $\bs=(\bs',s_k)\in(\P V)^{k-1}\times\P V$, where $\bs'=(s_1,\ldots,s_{k-1})\in(\P V)^{k-1}$, we have
\begin{align*}
\Gamma_{E,k}\setminus(&\Gamma_{E,k-1}\times\P V)\\ 
&=\{(x,\bs',s_k)\in X\times(\P V)^k:\,x\not\in D_{k-1}(\bs'),\,s_1(x)\wedge\ldots\wedge s_{k-1}(x)\wedge s_k(x)=0\} \\
&=\{(x,\bs'),s_k)\in\mathcal U\times\P V:\,s_k(x)\in\Span(s_1(x),\ldots,s_{k-1}(x))\}.
\end{align*}
Since $s_1(x),\ldots,s_{k-1}(x)\in E_x$ are linearly independent, it follows from (C) that 
\[\{s_k\in\P V:\,s_k(x)\in\Span(s_1(x),\ldots,s_{k-1}(x))\}\]
is a linear subspace of dimension $k-1+N-r$ of $\P V$. We infer that 
$\Gamma_{E,k}\setminus(\Gamma_{E,k-1}\times\P V)$ is a projective bundle of rank 
$k-1+N-r$ over the connected open set $\mathcal U\subset X\times(\P V)^{k-1}$, hence 
a connected complex manifold of dimension 
\[(k-1)N+n+k-1+N-r=kN+n+k-r-1.\]
By the induction hypothesis, $\Gamma_{E,k-1}\times\P V\subset\Gamma_{E,k}$ has dimension $kN+n+k-r-2$. 
Therefore $\Gamma_{E,k}$ is an irreducible analytic subset of $X\times(\P V)^k$ of dimension $kN+n+k-r-1$.

It remains to show that the restrictions of the canonical projections $\pi_{1,k},\pi_{2,k}$ to 
$\Gamma_{E,k}$ are surjective. Indeed, using (B), if $\bs=(s_1,\ldots,s_k)\in(\P V)^k$ and 
$x\in Z_{s_k}$ then $(x,\bs)\in\Gamma_{E,k}$. Condition (C) implies that for any $x\in X$ 
there exists $s\in\P V$ with $s(x)=0$, so $(x,\bs)\in\Gamma_{E,k}$, where $\bs=(\bs',s)$ 
for some $\bs'\in(\P V)^{k-1}$.
\end{proof}

Let $I_2(F_{E,k})=\{\bs\in(\P V)^k:\,\dim D_k(\bs)>n+k-r-1\}$ denote the second indeterminacy 
set of $F_{E,k}$. Then $\codim I_2(F_{E,k})\geq2$, and for $\bs\in(\P V)^k\setminus I_2(F_{E,k})$, 
the current of integration $[D_k(\bs)]$ has bidegree $(r+1-k,r+1-k)$. We let $\delta_\bs$ be Dirac 
mass at $\bs\in(\P V)^k$, viewed as a probability measure on $(\P V)^k$.

By \cite[Section 3.1]{DS06}, the pull-back of a current $T$ on $(\P V)^k$ of 
bidegree $(m,m)$, where $kN+k-r-1\leq m\leq kN$, is defined by 
\begin{equation}\label{e:pullback0}
F_{E,k}^\star(T)=(\pi_{1,k})_\star\big(\pi_{2,k}^\star(T)\wedge[\Gamma_{E,k}]\big),
\end{equation} 
in the case when $T$ is smooth, or when $T=[H]$ is the current of integration over an analytic 
subset $H\subset(\P V)^k$ of pure dimension $kN-m$ such that 
\[\dim(\pi_{2,k}|_{\Gamma_{E,k}})^{-1}(H\cap I_2(F_{E,k}))\leq kN-m+n+k-r-2.\]
Then $F_{E,k}^\star(T)$ is a current of bidegree $(m-kN+r+1-k,m-kN+r+1-k)$ on $X$. If $T$ is smooth, 
then $F_{E,k}^\star(T)$ is given by a form with coefficients in $L^1(X)$. If $T=[H]$ then 
$(\pi_{2,k}|_{\Gamma_{E,k}})^\star([H])$ is defined to be the current of integration 
$[(\pi_{2,k}|_{\Gamma_{E,k}})^{-1}(H)]$ over the analytic subset 
$(\pi_{2,k}|_{\Gamma_{E,k}})^{-1}(H)=\pi_{2,k}^{-1}(H)\cap\Gamma_{E,k}$ 
which has pure dimension $kN-m+n+k-r-1$. Its push-forward 
$F_{E,k}^\star(T)=(\pi_{1,k}|_{\Gamma_{E,k}})_\star[(\pi_{2,k}|_{\Gamma_{E,k}})^{-1}(H)]$ 
is a positive closed current on $X$ supported in the analytic set 
\[F_{E,k}^{-1}(H):=\pi_{1,k}\big(\pi_{2,k}^{-1}(H)\cap\Gamma_{E,k}\big).\]
Similarly, if $S$ is a smooth $(m,m)$ form on $X$, where $n+k-r-1\leq m\leq n$, then the push-forward 
\[(F_{E,k})_\star(S)=(\pi_{2,k})_\star\big(\pi_{1,k}^\star(S)\wedge[\Gamma_{E,k}]\big)\]
is a well defined current of bidegree $(m-n+r+1-k,m-n+r+1-k)$ on $(\P V)^k$ and it is given by a form with 
coefficients in $L^1((\P V)^k)$.

We endow $(\P V)^k$ with the K\"ahler form 
\begin{equation}\label{e:PVk-K}
\omega_k=c_k\sum_{j=1}^k\omega_\FS^j\,,\,\;\omega_\FS^j={\mathbf p}_j^\star\omega_\FS,
\end{equation}
where $\omega_\FS$ is the Fubini-Study form on $\P V$ and ${\mathbf p}_j:(\P V)^k\to\P V$ is the canonical projection on the $j$-th factor. Here the constant $c_k$ is chosen so that $\int_{(\P V)^k}\omega_k^{kN}=1$, so  $\omega_k^{kN}$ determines a probability measure $\Upsilon_k$ on $(\P V)^k$. We have 
\begin{equation}\label{e:PVk1}
\omega_k^{kN}=c_k^{kN}\,\frac{(kN)!}{(N!)^k}\,(\omega_\FS^1)^N\wedge\ldots\wedge(\omega_\FS^k)^N\,,\,\text{ so } c_k=\frac{(N!)^{1/N}}{((kN)!)^{1/kN}}\,.
\end{equation}
Using Stirling's formula we obtain that (see \cite[Lemma 3.2(i)]{CLMM})
\begin{equation}\label{e:ck}
c_k=\frac{(N!)^{1/N}}{((kN)!)^{1/kN}}\in\Big(\frac{1}{2ek}\,,\frac{2e}{k}\Big).
\end{equation}
We also have 
\begin{equation}\label{e:PVk2}
\begin{split}
\omega_k^{kN-1}&=c_k^{kN-1}\,\frac{(kN-1)!}{(N!)^{k-1}(N-1)!}\,
\sum_{j=1}^k(\omega_\FS^1)^N\wedge\ldots\wedge(\omega_\FS^j)^{N-1}\wedge\ldots
\wedge(\omega_\FS^k)^N\\
&=\frac{1}{kc_k}\,
\sum_{j=1}^k(\omega_\FS^1)^N\wedge\ldots\wedge(\omega_\FS^j)^{N-1}\wedge
\ldots\wedge(\omega_\FS^k)^N.
\end{split}
\end{equation}
%===
The top two intermediate degrees of $F_{E,k}$ are defined by 
\begin{align}\label{e:deg1}
\delta_1(F_{E,k})&=\int_XF_{E,k}^\star(\omega_k^{kN})
\wedge\omega^{n+k-r-1}=
\int_{(\P V)^k}\omega_k^{kN}\wedge(F_{E,k})_\star(\omega^{n+k-r-1}) \\
&=\int_{\Gamma_{E,k}}\pi_{1,k}^\star(\omega^{n+k-r-1})
\wedge\pi_{2,k}^\star(\omega_k^{kN})\,, \nonumber
\end{align}
\begin{align}\label{e:deg2}
\delta_2(F_{E,k})&=\int_XF_{E,k}^\star(\omega_k^{kN-1})
\wedge\omega^{n+k-r}=
\int_{(\P V)^k}\omega_k^{kN-1}\wedge(F_{E,k})_\star(\omega^{n+k-r}) \\
&=\int_{\Gamma_{E,k}}\pi_{1,k}^\star(\omega^{n+k-r})
\wedge\pi_{2,k}^\star(\omega_k^{kN-1})\,, \nonumber
\end{align}
(see \cite[(3.1)]{DS06}). We show next that they can be expressed 
in terms of certain Chern classes of $E$.

\begin{Theorem}\label{T:MTGdeg}
Let $(X,\omega),\,(E,h^E)$ verify assumptions (A)-(C), and assume that 
the Kodaira map $\Phi_E:X\to\G(r,V^\star)$ defined in \eqref{e:Kod2} 
is an embedding. If $1\leq k\leq r$ and $F_{E,k}$ is the meromorphic transform 
defined in \eqref{e:Gamma_k} then the following hold:
\begin{align}
F_{E,k}^\star(\delta_\bs)&=[D_k(\bs)], \text{ where } 
\bs\in(\P V)^k\setminus I_2(F_{E,k}),\label{e:MTp1}\\
F_{E,k}^\star(\omega_k^{kN})&=
\Phi_E^\star\big(c_{r+1-k}(\cT^\star,h^{\cT^\star})\big),\label{e:MTp2}\\
\delta_1(F_{E,k})&=\int_Xc_{r+1-k}(E)\wedge\omega^{n+k-r-1},\label{e:deg1c}\\
\delta_2(F_{E,k})&=\frac{1}{c_k}\int_Xc_{r-k}(E)\wedge
\omega^{n+k-r}.\label{e:deg2c}
\end{align}
\end{Theorem}

\begin{proof} Note that $\pi_{1,k}:
\pi_{2,k}^{-1}(\bs)\cap\Gamma_{E,k}=D_k(\bs)\times\{\bs\}\to D_k(\bs)$
is biholomorphic, and if $\bs\in(\P V)^k\setminus I_2(F_{E,k})$ then $D_k(\bs)$ 
has pure dimension $n+k-r-1$. So  \eqref{e:MTp1} follows from these observations 
and from \eqref{e:pullback0}.

If $\phi$ is an $(n+k-r-1,n+k-r-1)$ form on $X$ then  
\[\big\langle F_{E,k}^\star(\omega_k^{kN}),\phi\big\rangle=
\int_{\Gamma_{E,k}}\pi_{1,k}^\star(\phi)\wedge\pi_{2,k}^\star(\omega_k^{kN})=
\int_{(\P V)^k}(F_{E,k})_\star(\phi)\wedge\omega_k^{kN},\]
and $(F_{E,k})_\star(\phi)\in L^1((\P V)^k)$. For $\bs\in(\P V)^k\setminus I_2(F_{E,k})$ we have 
\[(F_{E,k})_\star(\phi)(\bs)=(\pi_{2,k})_\star\big(\big(\pi_{1,k}|_{\Gamma_{E,k}}\big)^\star\phi\big)(\bs)=
\int_{\pi_{2,k}^{-1}(\bs)\cap\Gamma_{E,k}}\pi_{1,k}^\star\phi=\int_{D_k(\bs)}\phi.\]
Since $\omega_k^{kN}=\mu_k$ is the product measure on $(\P V)^k$
determined by the Fubini-Study volume $\omega_\FS^N$ on $\P V$ (see \eqref{e:PVk1}) we obtain
\begin{equation}\label{e:MTexp}
\big\langle F_{E,k}^\star(\omega_k^{kN}),\phi\big\rangle=
\int_{(\P V)^k}\langle[D_k(\bs)],\phi\rangle\,\omega_k^{kN}=
\langle \E_k[\mu_k],\phi \rangle,
\end{equation}
where $\E_k[\mu_k]$ is the expectation current from \eqref{e:expk}.
Hence \eqref{e:MTp2} follows from Theorem \ref{T:expdeg}. 
For generic $\bs\in(\P V)^k$ the Chern class $c_{r+1-k}(E)$ is the Poincar\'e dual of $D_k(\bs)$ \cite[p.\ 413]{GH94}. 
Since $\omega$ is a closed form, this implies that
\[\int_{D_k(\bs)}\omega^{n+k-r-1}=\int_Xc_{r+1-k}(E)\wedge\omega^{n+k-r-1},\]
and \eqref{e:deg1c} follows from \eqref{e:MTexp} applied with $\phi=\omega^{n+k-r-1}$. 

To prove \eqref{e:deg2c} we use a cohomological argument. Since $\codim I_2(F_{E,k})\geq2$, we can find, for each 
$1\leq j\leq k$, points 
\[\bs_j=(s_1^j,\ldots,s_k^j)\in(\P V)^k\setminus I_2(F_{E,k}),\;\sigma_j\in\P V\setminus\Span(s_1^j,\ldots,s_k^j),\]
with the following property: if $L_j=\{as_j^j+b\sigma_j\in\P V:\,[a:b]\in\P^1\}$ is the line in $\P V$ determined by 
$s_j^j,\sigma_j$, and $A_j=\{s_1^j\}\times\ldots\times L_j\times\ldots\times\{s_k^j\}$, then $A_j\cap I_2(F_{E,k})=\emptyset$.

By \eqref{e:PVk2} we have the cohomology of bidimension $(1,1)$ currents,
\begin{equation}\label{e:cohom}
\omega_k^{kN-1}\sim\frac{1}{kc_k}\,\sum_{j=1}^k[A_j],
\end{equation}
where $[A_j]$ denotes the current of integration along the analytic set $A_j$. 
The current $F_{E,k}^\star([A_j])$ is well defined since $A_j\cap I_2(F_{E,k})=\emptyset$. Note that 
\[F_{E,k}^\star([A_j])=(\pi_{1,k}|_{\Gamma_{E,k}})_\star[\pi_{2,k}^{-1}(A_j)\cap\Gamma_{E,k}]\,,\,\;\pi_{2,k}^{-1}(A_j)\cap\Gamma_{E,k}=\bigcup_{\bs\in A_j}\big(D_k(\bs)\times\{\bs\}\big).\]
Moreover,
\[\bigcup_{\bs\in A_j}D_k(\bs)=\bigcup_{[a:b]\in\P^1}D_k(s_1^j,\ldots,as_j^j+b\sigma_j,\ldots,s_k^j)=D_{k+1}(\bs_j,\sigma_j).\]
Let $\mathcal A_j=D_k(\bs_j)\cap D_k(s_1^j,\ldots,\sigma_j,\ldots,s_k^j)$. We claim that 
\begin{equation}\label{e:bij}
\pi_{1,k}:(\pi_{2,k}^{-1}(A_j)\cap\Gamma_{E,k})\setminus\pi_{1,k}^{-1}(\mathcal A_j)\to 
D_{k+1}(\bs_j,\sigma_j)\setminus \mathcal A_j
\end{equation}
is bijective. Indeed, for $[a:b]\neq[a':b']$ we have 
\[D_k(s_1^j,\ldots,as_j^j+b\sigma_j,\ldots,s_k^j)\cap D_k(s_1^j,\ldots,a's_j^j+b'\sigma_j,\ldots,s_k^j)=\mathcal A_j.\]
This implies that for every $x\in D_{k+1}(\bs_j,\sigma_j)\setminus\mathcal A_j$ there exists a unique $[a:b]\in\P V$ such that $x\in D_k(s_1^j,\ldots,as_j^j+b\sigma_j,\ldots,s_k^j)$, which proves our claim. 

We infer from \eqref{e:bij} that $F_{E,k}^\star([A_j])=[D_{k+1}(\bs_j,\sigma_j)]$
is a current of bidegree $(r-k,r-k)$ that belongs to the Chern class $c_{r-k}(E)$. Hence
\[\int_XF_{E,k}^\star([A_j])\wedge\omega^{n+k-r}=\int_Xc_{r-k}(E)\wedge\omega^{n+k-r},\]
and \eqref{e:deg2c} follows from \eqref{e:deg2} and \eqref{e:cohom}.
\end{proof}

We conclude this section by recalling certain notions from pluripotential theory on multiprojective spaces that will be needed for the proof of Theorem \ref{T:distrib}. Let us identify $\P V=\P^N$ and consider the multiprojective space $\P^{N,k}=(\P^N)^k$ endowed with the K\"ahler form $\omega_k$ defined in \eqref{e:PVk-K}, \eqref{e:PVk1}. We denote by $PSH(\P^{N,k},\omega_k)$ the class of $\omega_k$-plurisubharmonic functions. These are the functions $\varphi$ on $\P^{N,k}$ which are locally the sum of a plurisubharmonic (psh) function and a smooth one, such that 
$\omega_k+dd^c\varphi\geq0$ in the sense of currents. Let 
\[\mathcal P_{N,k}=\Big\{\varphi\in PSH(\P^{N,k},k\omega_k):\,\int_{\P^{N,k}}\varphi\,\omega_k^{kN}=0\Big\}.\]
The following quantities related to $k\omega_k$-psh functions were introduced in \cite{DS06}:
\begin{align}
R_{N,k}=R(\P^{N,k},\omega_k,\omega_k^{kN})&=\sup\Big\{\max_{\P^{N,k}}\varphi:\,\varphi\in\mathcal P_{N,k}\Big\}, \label{e:R}\\
\Delta_{N,k}(t)=\Delta(\P^{N,k},\omega_k,\omega_k^{kN},t)&= \sup\Big\{\int_{\{\varphi<-t\}}\omega_k^{kN}:\,
\varphi\in\mathcal P_{N,k}\Big\},\;t\in\R. 
\label{e:Delta}
\end{align}
Note that the constant $r(\P^{N,k},\omega_k)$ introduced in 
\cite[Proposition 2.2, (2.1)]{DS06} satisfies $r(\P^{N,k},\omega_k)=k$ 
(see \cite[Proposition A.9]{DS06}, \cite[Lemma 4.6]{CMN16}). 
This explains the considerations of $k\omega_k$-psh functions in the above definitions. 
It is shown in  \cite[Proposition A.9]{DS06} that there exist constants 
$c>0,\alpha>0,\zeta>0$ depending only on $k$ such that 
\begin{equation}\label{e:RDest}
R_{N,k}\leq c(1+\log N)\,,\,\;\Delta_{N,k}(t)
\leq cN^\zeta e^{-\alpha t},\;\forall\,t>0.
\end{equation}

%\section{Proof of Theorem \ref{T:distrib}}\label{S:distrib}
\section{Equidistribution of random degeneracy sets}\label{S:distrib}

In this Section, we give the proof of Theorem \ref{T:distrib} 
by using Theorem \ref{T:Tian} in conjunction 
with an equidistribution theorem of Dinh and Sibony for meromorphic 
transforms \cite{DS06}. We will apply their result 
to a sequence of meromorphic transforms associated to degeneracy sets 
defined as in \eqref{e:Gamma_k}.

\begin{Theorem}\label{T:distrib2}
Let $(X,\omega)$ be a compact K\"ahler manifold of dimension $n$, 
$(L,h^L)$ be a positive line bundle on $X$, and $(E,h^E)$ 
be a Hermitian holomorphic vector bundle of rank $r\leq n$ on $X$. 
Let $(\mathcal H_k,\Upsilon_k)$ be the probability space defined in \eqref{e:prob}. 
Then there exist $C>0$ and $p_0\in\N$ such that the following holds: 
For any $1\leq k\leq r$ and $\gamma>1$ there exist subsets 
$\mathcal{E}_{p,k}=\mathcal{E}_{p,k}(\gamma)\subset(\P V_p)^k$ such that for $p>p_0$ we have 

(i) $\Upsilon_{p,k}(\mathcal{E}_{p,k})\leq Cp^{-\gamma}$;

(ii) if $\bs_p\in(\P V_p)^k\setminus \mathcal{E}_{p,k}$ then
\[\Big|\frac{1}{p^{r+1-k}}\,\Big\langle[D_k(\bs_p)]-\Phi_p^\star(c_{r+1-k}(\cT^\star,h^{\cT^\star})),
\phi\Big\rangle\Big|\leq C\gamma\,\frac{\log p}{p}\,\|\phi\|_{\cC^2(X)},\]
for any $(n+k-r-1,n+k-r-1)$ form $\phi$ of class $\cC^2$ on $X$.
Moreover, the estimate (ii) holds for $\Upsilon_k$-a.e. sequence
 $\{\bs_p\}_{p\geq1}\in\mathcal H_k$ provided
that $p$ is large enough.
\end{Theorem}

\begin{proof}
Recall the definition \eqref{e:Kod1} of the Kodaira maps $\Phi_p$. Since $(L,h^L)$ is positive 
we have, for all $p$ sufficiently large, that the vector bundles $L^p\otimes E$ verify (C) 
\cite[Theorem 5.1.15]{MM07} and that $\Phi_p$ is an embedding \cite[Theorem 5.1.18]{MM07}.
We showed in \cite{BCLM} that the vector bundles $L^p\otimes E$ 
also verify (B) (see \cite[Theorem 6.1]{BCLM} and its proof).

For $1\leq k\leq r$, we consider the meromorphic transforms $F_{p,k}$ from $X$ to $(\P V_p)^k$ defined as in \eqref{e:Gamma_k}, with graph
\begin{equation}\label{e:Gamma_pk}
\Gamma_{p,k}=\{(x,\bs_p)\in X\times(\P V_p)^k:\,x\in D_k(\bs_p)\}\subset X\times(\P V_p)^k.
\end{equation}
By Theorem \ref{T:MTD}, we have, for each $p$ sufficiently large, that $\Gamma_{p,k}$ is an irreducible analytic subset of $X\times(\P V_p)^k$ of dimension $kd_p+n+k-r-1$, and $F_{p,k}$ is a meromorphic transform of codimension $n+k-r-1$.  
%Let $\pi_{1,k}:X\times(\P V_p)^k\to X$, $\pi_{2,k}:X\times(\P V_p)^k\to(\P V_p)^k$ be the canonical projections.
Moreover, for $p$ large enough (see, e.g., \cite[(2.9)]{BCLM}),
\begin{equation}\label{e:dp}
d_p=\dim V_p-1= r\int_X \frac{c_1(L)^n}{n!} \;p^n +R_{n-1}(p),
\end{equation}
where $R_{n-1}(p)$ is a polynomial of degree $(n-1)$ in $p$. We endow $(\P V_p)^k$ with the K\"ahler form $\omega_{p,k}$ defined as in \eqref{e:PVk-K} with constant (see \eqref{e:PVk1}, \eqref{e:ck})
\begin{equation}\label{e:cpk}
c_{p,k}=\frac{(d_p!)^{1/{d_p}}}{((kd_p)!)^{1/kd_p}}\in\Big(\frac{1}{2ek}\,,\frac{2e}{k}\Big).
\end{equation}
Then $\omega_{p,k}^{kd_p}=\Upsilon_{p,k}$ is the probability measure on $(\P V_p)^k$ considered in Theorems \ref{T:distrib} and \ref{T:distrib2}.

Let 
\[R_{p,k}=R((\P V_p)^k,\omega_{p,k},\omega_{p,k}^{kd_p}),\;
\Delta_{p,k}(t)=\Delta((\P V_p)^k,\omega_{p,k},\omega_{p,k}^{kd_p},t),\;t\in\R, \]
be the pluripotential-theoretic quantities from \eqref{e:R}, \eqref{e:Delta}, corresponding to the K\"ahler manifold 
$((\P V_p)^k,\omega_{p,k})$. By \eqref{e:dp} we have $d_p\sim p^n$. Using \eqref{e:RDest}, we infer that there exist constants $C_1>0$, $p_0\in\N$, and $\alpha>0,\zeta>0$ depending only on $k$, such that
\begin{equation}\label{e:RDestp}
R_{p,k}\leq C_1(1+n\log p),\;\Delta_{p,k}(t)\leq C_1p^{n\zeta}e^{-\alpha t}, \text{ for }p>p_0,\;t>0.
\end{equation}

Recall that (see, e.g., \cite[Lemma 2.2]{BCLM})
\[c_j(L^{p}\otimes E,h^{L^{p}\otimes E})=p^j\binom{r}{j} c_1(L,h^L)^j+O(p^{j-1}).\]
Using this together with Theorem \ref{T:MTGdeg} and \eqref{e:cpk}, we obtain the asymptotics of the top intermediate degrees of $F_{p,k}$:
\begin{align}
\delta_1(F_{p,k})&=\int_Xc_{r+1-k}(L^{p}\otimes E,h^{L^{p}\otimes E})\wedge\omega^{n+k-r-1}\label{e:deg1cp}\\ 
&=p^{r+1-k}\binom{r}{k-1}\int_Xc_1(L,h)^{r+1-k}\wedge\omega^{n+k-r-1}+O(p^{r-k}),\nonumber\\
\delta_2(F_{p,k})&=\frac{1}{c_{p,k}}\int_Xc_{r-k}(L^{p}\otimes E,h^{L^{p}\otimes E})\wedge\omega^{n+k-r}\label{e:deg2cp}\\
&=\frac{p^{r-k}}{c_{p,k}}\,\binom{r}{k}\int_Xc_1(L,h)^{r-k}\wedge\omega^{n+k-r}+O(p^{r-k-1}).\nonumber
\end{align}
Increasing if necessary the constants $C_1,p_0$ and applying \eqref{e:cpk}, we get
\begin{equation}\label{e:ida}
C_1^{-1}p^{r+1-k}\leq\delta_1(F_{p,k})\leq C_1p^{r+1-k},\;
C_1^{-1}p\leq\frac{\delta_1(F_{p,k})}{\delta_2(F_{p,k})}\leq C_1p, \text{ for } p>p_0.
\end{equation}

For $\varepsilon>0$ let
\[\mathcal{E}'_{p,k}(\varepsilon):=\bigcup_{\|\phi\|_{\cC^2(X)}\leq 1}
\left\{\bs_p\in(\P V_p)^k:\,\left|\left\langle F_{p,k}^\star(\delta_{\bs_p}) - 
F_{p,k}^\star(\Upsilon_{p,k}),\phi\right\rangle\right| 
\geq \delta_1(F_{p,k})\varepsilon \right\},\]
\[\mathcal{E}''_{p,k}(\varepsilon):=\bigcup_{\|\phi\|_{\cC^2(X)}\leq 1}
\left\{\bs_p\in(\P V_p)^k:\,\left|\left\langle F_{p,k}^\star(\delta_{\bs_p}) - 
F_{p,k}^\star(\Upsilon_{p,k}),\phi\right\rangle\right|  
\geq C_1p^{r+1-k}\varepsilon \right\},\]
where $\phi$ is a $(n+k-r-1,n+k-r-1)$ form on $X$ of class $\cC^2$. 
Note that by \eqref{e:ida}, $\mathcal{E}''_{p,k}(\varepsilon)\subset\mathcal{E}'_{p,k}(\varepsilon)$. 
The Dinh-Sibony equidistribution theorem for meromorphic transforms \cite[Lemma 4.2\,(d)]{DS06}
states that, for all $\varepsilon>0$, we have 
\begin{equation}\label{e:est2}
\Upsilon_{p,k}(\mathcal{E}''_{p,k}(\varepsilon))\leq
\Upsilon_{p.k}(\mathcal{E}'_{p,k}(\varepsilon))\leq
\Delta_{p,k}(t_{\varepsilon,p,k}), \text{ where } t_{\varepsilon,p,k}=
\varepsilon\,\frac{\delta_1(F_{p,k})}{\delta_2(F_{p,k})}-3R_{p,k}.
\end{equation}
Using \eqref{e:ida} and \eqref{e:RDestp} we get, for all $p>p_0$, 
\[t_{\varepsilon,p,k}\geq\frac{\varepsilon p}{C_1}-3C_1(1+n\log p).\]
Combining this with the estimate on $\Delta_{p,k}$ from \eqref{e:RDestp}, we obtain from \eqref{e:est2} that 
\[\Upsilon_{p,k}(\mathcal{E}''_{p,k}(\varepsilon))\leq 
C_1p^{n\zeta}\exp\Big(-\frac{\alpha\varepsilon p}{C_1}+3\alpha C_1+3\alpha C_1n\log p\Big),\;\forall\,\varepsilon>0.\]

We fix now $\gamma>1$ and choose 
$\varepsilon=\varepsilon_{p,k,\gamma}$ such that
\[-\frac{\alpha\varepsilon p}{C_1}+3\alpha C_1+3\alpha C_1n\log p=-(n\zeta+\gamma)\log p.\]
Let $\mathcal{E}_{p,k}=\mathcal{E}_{p,k}(\gamma):=
\mathcal{E}''_{p,k}(\varepsilon_{p,k,\gamma})$. 
Then for all $p>p_0$ we have that
\[\Upsilon_{p,k}(\mathcal{E}_{p,k}(\gamma))\leq C_1p^{-\gamma}.\]
By \eqref{e:MTp2} we have 
\[F_{p,k}^\star(\Upsilon_{p,k})=F_{p,k}^\star(\omega_{p,k}^{kd_p})=\Phi_p^\star\big(c_{r+1-k}(\cT^\star,h^{\cT^\star})\big).\]
Moreover, by \eqref{e:MTp1}, $F_{p,k}^\star(\delta_{\bs_p})=[D_k(\bs_p)]$, for $\bs_p\in(\P V_p)^k\setminus I_2(F_{p,k})$. 
If $\bs_p\in(\P V_p)^k\setminus \mathcal{E}_{p,k}$ and 
$\phi$ is a $(n+k-r-1,n+k-r-1)$ form on $X$ 
of class $\cC^2$, we obtain by the definition of $\mathcal{E}_{p,k}$ that 
\[\Big|\frac{1}{p^{r+1-k}}\,\Big\langle[D_k(\bs_p)]-
\Phi_p^\star(c_{r+1-k}(\cT^\star,h^{\cT^\star})),\phi\Big\rangle\Big|
\leq C_1\varepsilon_{p,k,\gamma}\|\phi\|_{\cC^2(X)}.\]
Note that 
\[\varepsilon_{p,k,\gamma}=\frac{C_1}{\alpha p}
\big((3\alpha C_1n+n\zeta+\gamma)\log p+3\alpha C_1\big)\leq C_2\gamma\,\frac{\log p}{p}\]
holds for $p>p_0$, where $C_2>0$ is a constant independent of $p$ and $\gamma$. 
Assertion $(ii)$ of Theorem \ref{T:distrib2} now follows. Finally, the Borel-Cantelli lemma yields the last 
assertion of Theorem \ref{T:distrib2}, since $\sum_{p=1}^\infty\Upsilon_{p,k}(\mathcal{E}_{p,k})<\infty$.
\end{proof}

\begin{proof}[Proof of Theorem \ref{T:distrib}]
Theorem \ref{T:distrib} follows at once from Theorems \ref{T:distrib2} and \ref{T:Tian}.
\end{proof}

%======= 08 march 2025
%\begin{Remark}\label{Rk:6.2}
%Note that in Theorem \ref{T:zeros}, we may always replace $(\P V_p,\Upsilon_p)$ by $(V_p, \mathcal{G}_p)$, where $\mathcal{G}_p$ denotes the Gaussian probability measure on $V_p$ induced by the $L^2$-inner product. Indeed, if $\pi:V_p\setminus\{0\}\to\P V_p$ is the canonical projection and $\wi{\mathcal{E}_p}=\pi^{-1}(\mathcal{E}_p)$ then $\mathcal{G}_p(\wi{\mathcal{E}_p})=\Upsilon_p(\mathcal{E}_p)$, so Theorem \ref{T:zeros} holds for the sets $\wi{\mathcal{E}_p}$.
%This follows from a simple observation that, via the canonical projection $V_p\setminus\{0\}\rightarrow \P V_p$, the random $(r,r)$-current $[s_p=0]$ modeled by $(\P V_p,\Upsilon_p)$ has the same distribution as when modeled by $(V_p, \mathcal{G}_p)$.
%\end{Remark}

%June 23rd 2026
\section{Wishart distribution and the determinant case}\label{S:Wick}
In this final section, we discuss a complementary proof of the equidistribution result in
the special case \(k=r\). In this case, the degeneracy locus is a hypersurface: it is the
zero divisor of the wedge product
\(
s^p_{1}\wedge\cdots\wedge s^p_{r}\in H^0(X,L^{pr}\otimes \det E).
\)
Thus, the problem can be reduced to the study of random holomorphic sections of the determinant
line bundle \(L^{pr}\otimes\det E\). For Gaussian sections, the pointwise norm of the section $s^p_{1}\wedge\cdots\wedge s^p_{r}$ can be calculated explicitly with the help of the complex Wishart distribution. This
gives a local probabilistic route to the convergence of the normalized currents
\(
\frac1p[D_{r}(s^p_{1},\cdots, s^p_{r})]
\)
towards \(rc_1(L,h^L)\). 
Although this argument only treats the determinant case, it is useful because it gives a
direct local proof based on the pointwise Gaussian covariance 
\(P_p(x)\) and provides a
comparison with the general meromorphic-transform method used above.

We start by reviewing several fundamental results about
the complex Wishart and chi‑square distributions.
Let $\xi^\top=(Z_1, Z_2, \ldots, Z_r)$ be a $r$-tuple 
of symmetric complex Gaussian variables with mean zero. 
Set the covariance matrix
\begin{equation}
\Sigma_{\xi}=(\Sigma_{\xi,ij})_{1\leq i,j\leq r} :=\E[\xi\ov{\xi}^\top]\,,\quad \Sigma_{\xi,ij}:=\E[Z_i\ov{Z}_j].
\end{equation}
We always assume that $\Sigma_{\xi}$ is nonsingular and define the generalized variance as
\begin{equation}
\sigma_{\xi}:=\det \Sigma_{\xi} >0.
\end{equation}

\begin{Definition}[Complex Wishart distribution (see \cite{Goodman_1963a, Goodman_1963b, MR0652932})]\label{D:Wishart}
Fix $m\in\N_{\geq 1}$ such that $m\geq r$, let $\xi_1,\ldots,\xi_m$ 
denote $m$ independent and identically distributed $r$-tuples of 
mean-zero symmetric complex Gaussian variables with Hermitian 
covariance matrix $\Sigma_\xi$ as above. A random matrix of 
complex Wishart distribution is defined as
\begin{equation}
\widehat{\boldsymbol{\Sigma}}_\xi:=\frac{1}{m}\sum_{j=1}^m \xi_j\ov{\xi}_j^\top.
\label{e:5.3Wishart}
\end{equation}
\end{Definition}

\begin{Definition}
Let $N_1, N_2, \dots, N_k$ be independent standard real 
Gaussian variables, i.e., $N_j \sim \mathcal{N}_\R(0,1)$ 
for $j=1,\dots,k$. Then the random variable
\[
\eta = \sum_{j=1}^{k} N_j^2
\]
is said to have a \emph{chi‑square distribution} 
with $k$ degrees of freedom, denoted by
\[
\eta \sim \chi^2(k).
\]
Note that we have $\mathbb{E}[X] = k$ and $\operatorname{Var}[X] = 2k$.
\end{Definition}

A key feature of the complex Wishart distribution is 
Bartlett’s decomposition theorem, which can be derived by applying 
the QR factorization to matrices.
%===
\begin{Theorem}[{\cite{W34, Goodman_1963a, Goodman_1963b}}] \label{T:Goodman}
Let $\widehat{\boldsymbol{\Sigma}}_\xi$ be the random variable as in \eqref{e:5.3Wishart}, then the random variable
$$(2m)^r\frac{\det \widehat{\boldsymbol{\Sigma}}_\xi}{\sigma_\xi}$$
has the same distribution as the product of $r$ independent chi-square random variables with $2m$, $2(m-1)$, $\cdots$, $2(m-r+1)$ degrees of freedom.
\end{Theorem}

%For chi‑square distribution, we have $\mathbb{E}[X] = k$ and $\operatorname{Var}(X) = 2k$.
%\begin{itemize}
%    \item Mean: $\mathbb{E}[X] = k$.
%    \item Variance: $\operatorname{Var}(X) = 2k$.
%    \item Moment generating function: $M_X(t) = (1 - 2t)^{-k/2}$ for $t < 1/2$.
%    \item Additivity: If $X_1 \sim \chi^2(k_1)$ and $X_2 \sim \chi^2(k_2)$ are independent, then $X_1 + X_2 \sim \chi^2(k_1 + k_2)$.
 %   \item Asymptotic normality: As $k \to \infty$, $(X - k)/\sqrt{2k} \xrightarrow{d} \mathcal{N}(0,1)$.
%\end{itemize}

We have  the following key result deduced from Theorem \ref{T:Goodman}, where the expectation formula \eqref{e:WG2}  is standard in the literature and can also be proved in a more elementary way without using Theorem \ref{T:Goodman}.
\begin{Corollary}[Wick--Gram identity]
\label{Cor:WG}
Let $(H,\langle\cdot,\cdot\rangle)$ be a Hermitian vector space of dimension $r$.
Let $v_1,\dots,v_r$ be i.i.d.\ centered (symmetric) complex Gaussian vectors in $H$ with a covariance operator
\[
\Sigma:=\mathbb E\bigl[v_1\otimes v_1^{*}\bigr]\in \End(H),
\qquad\text{i.e.}\quad
\mathbb E\langle u,v_1\rangle\,\overline{\langle w,v_1\rangle}=\langle u,\Sigma w\rangle.
\]
%Let $\det(\Sigma)$ is computed in any orthonormal basis of $H$.
Then there exist $r$ independent chi-square variables $Y_1$, $\cdots$, $Y_r$ with $Y_j\sim \chi^2(2j)$, $j=1,2,\cdots,r$, such that the random variable $\|v_1\wedge\cdots\wedge v_r\|^2$ has the same distribution as
\begin{equation}
\frac{\det \Sigma}{2^r}\Pi_{j=1}^r Y_j.
\label{e:WG1}
\end{equation}
In particular, we have
\begin{equation}
\mathbb E\left[\|v_1\wedge\cdots\wedge v_r\|^2\right]=r!\,\det \Sigma.
\label{e:WG2}
\end{equation}
\end{Corollary}

\begin{proof}
Fix an orthonormal basis $(e_1,\dots,e_r)$ of $H$ and write
$v_j=\sum_{\ell=1}^r \xi_{j\ell}e_\ell$ with complex Gaussian coefficients.
Then
\[
v_1\wedge\cdots\wedge v_r=\det(\Xi)\,e_1\wedge\cdots\wedge e_r,
\qquad
\Xi^\top:=(\xi_{j\ell})_{1\le j,\ell\le r},
\]
hence $\|v_1\wedge\cdots\wedge v_r\|^2=|\det(\Xi)|^2$.

Let $\xi_j=(\xi_{j1},\ldots,\xi_{jr})^\top$ denote the vertical Gaussian vector defined by $v_j$. Then the square matrix $\Xi$ is given as $r$ independent columns of Gaussian vectors $\xi_1, \ldots, \xi_r$ with the same probability distribution. A direct calculation shows
\begin{equation}
\begin{split}
\|v_1\wedge\cdots\wedge v_r\|^2 &=\det(\Xi\Xi^\star) \\
&=\det \begin{bmatrix} \xi_1 & \ldots & \xi_r \end{bmatrix} \begin{bmatrix} \ov{\xi}^\top_1 \\ \vdots \\  \ov{\xi}^\top_r\end{bmatrix}\\
&=\det\left(\sum_{j=1}^r \xi_j\ov{\xi}^\top_j\right)
\end{split}
\end{equation}

By Definition \ref{D:Wishart}, we conclude that $\frac{1}{r}\sum_{j=1}^r \xi_j\ov{\xi}^\top_j$ admits a complex Wishart distribution associated with the covariance $\Sigma$. Then \eqref{e:WG1} follows directly from Theorem \ref{T:Goodman}.

Since $Y_j\sim \chi^2(2j)$, we have $\E[Y_j]=2j$. The independence of $Y_j$ ($j=1,\ldots, r$) and \eqref{e:WG1} imply that 
\begin{equation}
\mathbb E\left[\|v_1\wedge\cdots\wedge v_r\|^2\right]=\frac{\det \Sigma}{2^r}\Pi_{j=1}^r \E[Y_j].
\label{e:WG3}
\end{equation}
Then we obtain \eqref{e:WG2}.

%Moreover, $\mathbb E(\xi_{j\ell}\overline{\xi_{jm}})=\Sigma_{\ell m}$ and the rows of $\Xi$ are independent.
%Expanding $|\det(\Xi)|^2$ as
%\[
%|\det(\Xi)|^2=\sum_{\sigma,\tau\in S_r}sgn(\sigma)sgn(\tau)
%\prod_{j=1}^r \xi_{j,\sigma(j)}\,\overline{\xi_{j,\tau(j)}},
%\]
%independence across $j$ gives
%\[
%\mathbb E|\det(\Xi)|^2=\sum_{\sigma,\tau\in S_r}sgn(\sigma)sgn(\tau)
%\prod_{j=1}^r \mathbb E\!\left(\xi_{j,\sigma(j)}\,\overline{\xi_{j,\tau(j)}}\right)
%=\sum_{\sigma,\tau\in S_r}sgn(\sigma)sgn(\tau)\prod_{j=1}^r \Sigma_{\sigma(j),\tau(j)}.
%\]
%With $\pi:=\tau\circ\sigma^{-1}$ one has
%\[
%\prod_{j=1}^r \Sigma_{\sigma(j),\tau(j)}=\prod_{k=1}^r \Sigma_{k,\pi(k)},
%\qquad
%sgn(\sigma)sgn(\tau)=sgn(\pi),
%\]
%so summing over $\sigma$ yields the factor $r!$ and
%\[
%\mathbb E|\det(\Xi)|^2
%=r!\sum_{\pi\in S_r}sgn(\pi)\prod_{k=1}^r \Sigma_{k,\pi(k)}
%=r!\,\det(\Sigma).
%\]
\end{proof}

%\begin{proof}
%Fix an orthonormal basis $(e_1,\dots,e_r)$ of $H$ and write
%$v_j=\sum_{\ell=1}^r \xi_{j\ell}e_\ell$ with complex Gaussian coefficients.
%Then
%\[
%v_1\wedge\cdots\wedge v_r=\det(\Xi)\,e_1\wedge\cdots\wedge e_r,
%\qquad
%\Xi:=(\xi_{j\ell})_{1\le j,\ell\le r},
%\]
%hence $\|v_1\wedge\cdots\wedge v_r\|^2=|\det(\Xi)|^2$.
%Moreover, $\mathbb E(\xi_{j\ell}\overline{\xi_{jm}})=\Sigma_{\ell m}$ and the rows of $\Xi$ are independent.
%Expanding $|\det(\Xi)|^2$ as
%\[
%|\det(\Xi)|^2=\sum_{\sigma,\tau\in S_r}sgn(\sigma)sgn(\tau)
%\prod_{j=1}^r \xi_{j,\sigma(j)}\,\overline{\xi_{j,\tau(j)}},
%\]
%independence across $j$ gives
%\[
%\mathbb E|\det(\Xi)|^2=\sum_{\sigma,\tau\in S_r}sgn(\sigma)sgn(\tau)
%\prod_{j=1}^r \mathbb E\!\left(\xi_{j,\sigma(j)}\,\overline{\xi_{j,\tau(j)}}\right)
%=\sum_{\sigma,\tau\in S_r}sgn(\sigma)sgn(\tau)\prod_{j=1}^r \Sigma_{\sigma(j),\tau(j)}.
%\]
%With $\pi:=\tau\circ\sigma^{-1}$ one has
%\[
%\prod_{j=1}^r \Sigma_{\sigma(j),\tau(j)}=\prod_{k=1}^r \Sigma_{k,\pi(k)},
%\qquad
%sgn(\sigma)sgn(\tau)=sgn(\pi),
%\]
%so summing over $\sigma$ yields the factor $r!$ and
%\[
%\mathbb E|\det(\Xi)|^2
%=r!\sum_{\pi\in S_r}sgn(\pi)\prod_{k=1}^r \Sigma_{k,\pi(k)}
%=r!\,\det(\Sigma).
%\]
%\end{proof}

Let $\nu^p_r$ denote the product of Gaussian probability measures on $H^0(X,L^p\otimes E)^r$ induced by the $L^2$-inner product. Take the Gaussian vector 
\begin{equation}
\bs_p=(s^p_1,\ldots, s^p_r)\in (H^0(X,L^p\otimes E)^r,\nu^p_r),
\label{e:random5}
\end{equation}
and, for sufficiently large $p$, we consider the random $(1,1)$-current $[D_r(\bs_p)]$ on $X$. We view $s^p_{1}\wedge\cdots\wedge s^p_{r}$ as a holomorphic section of the holomorphic line bundle $L^{pr}\otimes \det E$, by the Poincar\'{e}-Lelong formula, we have
\begin{equation}
    [D_r(\bs_p)]=p\, r c_1(L,h^L)+c_1(E,h^E)+\frac{\sqrt{-1}}{2\pi}\partial\ov{\partial}\log \|s^p_1\wedge\cdots\wedge s^p_r\|^2_{h^{pr}\otimes h^{\det E}}(x),
    \label{e:PL}
\end{equation}
where $c_1(E,h^E)=c_1(\det E, h^{\det E})$.
Therefore, it is important to understand the properties of the random variable $\log \|s^p_1\wedge\cdots\wedge s^p_r\|^2_{h^{pr}\otimes h^{\det E}}(x)$.

Recall that when $p$ is sufficiently large, for every $x \in X$, the map $P_p(x) \in \End(E_x)$ is a Hermitian, positive-definite endomorphism of the Hermitian vector space $(E_x, h_x^E)$.
%===
\begin{Proposition}[Pointwise characterization of determinant section]\label{P:WG-section}
We consider the random section $\bs_p=(s^p_1,\ldots, s^p_r) \in H^0(X,L^p\otimes E)^r $ as in \eqref{e:random5}. For any $x\in X$, the random variable
$\|s^p_1\wedge\cdots\wedge s^p_r\|^2_{h^{pr}\otimes h^{\det E}}(x)$ admits the decomposition
\begin{equation}
\frac{1}{2^r}\det P_p(x) \prod_{\ell=1}^r \chi^2(2\ell),
\label{e:5.9decomp}
\end{equation}
where the chi-square random variables are independent of each other for any given $x\in X$.
Moreover, we have
\begin{equation}
\E[\log \|s^p_1\wedge\cdots\wedge s^p_r\|^2_{h^{pr}\otimes h^{\det E}}(x)]=\log \det P_p(x) +\sum_{\ell=1}^r \psi(\ell),
\label{e:5.9exp}
\end{equation}
where $\psi(z):=\frac{\Gamma'(z)}{\Gamma(z)}$ denotes the digamma function, and $\psi(1)=-\gamma$ with $\gamma$ being the Euler–Mascheroni constant; for integer $\ell\geq 2$,
$\psi(\ell)=-\gamma+\sum_{m=1}^{\ell-1}\frac{1}{m}$.
\end{Proposition}
\begin{proof}
Actually, the proof can proceed at a given point $x$, but we will work on an open neighborhood of the given $x$ in order to use the same arguments when we calculate the differentials of the expectation on the left-hand side of \eqref{e:5.9exp}.

Let $\{S^p_j\}_{j=1}^{d_p+1}$ be an orthonormal basis of $H^0(X, L^p\otimes E)$. 
We fix a point $x$, and in a small open neighborhood $U_x$ of $x$, we fix 
a local smooth unit frame $e_L$ of $L$, and a local orthonormal smooth frame 
$\{v_1,\ldots v_r\}$ of $(E,h^E)$ (which are not necessarily holomorphic). 
Then the frame $\{e_L^{\otimes p}\otimes v_1, \ldots, e_L^{\otimes p}\otimes v_r\}$ 
is a local orthonormal basis of $(L^p\otimes E)_y$ for any point $y\in U_x$.
There exist smooth functions $f^p_{jm}$ on $U_x$ such that
$$S^p_j(y)=\sum_{m=1}^r f^p_{jm}(y) e_L^{\otimes p}\otimes v_m\qquad\text{on $U_x$}.$$
Then by our assumption, the section $P_p(y)\in\mathrm{End}(E_y)$ has the following matrix form with respect to the local orthonormal smooth frame $\{v_1,\ldots v_r\}$:
\begin{equation}
    P_p(y)=(P_p(y)_{km})_{1\leq k,m\leq r}\,,\, P_p(y)_{km}:=\sum^{d_p+1}_{j=1} f^p_{jk}(y)\ov{f}^p_{jm}(y).
    \label{e:5.10Pp}
\end{equation}
%===
Let $\{\eta_{\ell j}\}_{1\leq \ell\leq r, 1\leq j\leq d_p+1}$ be a sequence of independent and identically distributed standard complex Gaussian variables (of total variance $1$). Then we can write, for $\ell=1,\ldots, r$,
$$s^p_\ell(y)=\sum_{j=1}^{d_p+1}\eta_{\ell j}S^p_j(y)=\sum_{m=1}^r \sum_{j=1}^{d_p+1}\eta_{\ell j}f^p_{jm}(y) e_L^{\otimes p}\otimes v_m.$$
We define the Gaussian variable
$$\xi_{\ell m}(y):=  \sum_{j=1}^{d_p+1}\eta_{\ell j}f^p_{jm}(y).$$
Then it is clear that 
$\E[\xi_{\ell k}(y)\ov{\xi}_{\ell m}(y)]=P_p(y)_{km}$. 
We conclude that the random Gaussian vectors $s^p_1(y), \ldots, s^p_r(y)$ 
in $L^p_y\otimes E_y$ are independent and identically distributed with 
the covariance matrix given by $P_p(y)$ in \eqref{e:5.10Pp}, 
with respect to the orthonormal basis 
$\{e_L^{\otimes p}\otimes v_1, \ldots, e_L^{\otimes p}\otimes v_r\}$. 
We then invoke Corollary \ref{Cor:WG}, which yields precisely the decomposition in \eqref{e:5.9decomp}.

Finally, expression \eqref{e:5.9exp} is obtained from \eqref{e:5.9decomp} 
by invoking the identity
\begin{equation}
    \E[\log \chi^2(2\ell)] = \log 2 + \psi(\ell),
\end{equation}
where $\psi$ denotes the digamma function.
\end{proof}

By \cite[Corollary 1.5]{BCLM} and Theorem \ref{T:Tian}, the expectation of $[D_r(\bs_p)]$ exists and admits an expansion (see \eqref{e:Easym})
\begin{equation}
\begin{split}
     \frac{1}{p}\E[[D_r(\bs_p)]]&=rc_1(L,h^L)+\frac{1}{p}\left(c_1(E,h^E)+r\alpha_L\right)\\
     &\qquad+\frac{1}{p^2}\left(r\frac{\sqrt{-1}}{16\pi^2}\partial\ov{\partial} r^X_L + \frac{\sqrt{-1}}{2\pi}\partial\ov{\partial}\left(\Lambda_{\omega_L}c_1(E,h^E)+
     r\Lambda_{\omega_L}\alpha_L\right)\right)+O(p^{-3})
\end{split}
\label{e:randomE5}
\end{equation}
Now we provide a second proof of \eqref{e:randomE5} as an application of Proposition \ref{P:WG-section}.

\begin{proof}[Second proof of \eqref{e:randomE5}]
Combining \eqref{e:PL} and \eqref{e:5.9exp}, we have the following formula for the expectation current
\begin{equation}
\begin{split}
     \frac{1}{p}\E[[D_r(\bs_p)]]=&rc_1(L,h^L)+
     \frac{1}{p}c_1(E,h^E)+
     \frac{\sqrt{-1}}{2\pi p}\partial\ov{\partial}\log \det P_p(x).
\end{split}
\label{e:E5proof}
\end{equation}

For a square matrix $A$ whose eigenvalues are sufficiently small, 
the Taylor series expansion of the logarithm yields
\[
\log\det(I+A)
=\operatorname{Tr}[A]
+O\!\bigl(\|A\|^2\bigr).
\]
Consequently, for sufficiently large $p$, by Theorem \ref{bkt2.1} and \eqref{e:DeltaL}, we obtain
\begin{equation}
   \log \det P_p(x)= nr\log p +r\log b_0(x) +\frac{r}{8\pi p}r^X_L(x)+\frac{\sqrt{-1}}{2\pi p}\left(\Lambda_{\omega_L}c_1(E,h^E)+r\Lambda_{\omega_L}\alpha_L\right)+O(p^{-2}),
   \label{e:5.16proof}
\end{equation}
where the expansion holds in $\mathscr{C}^\infty$-topology. Then \eqref{e:randomE5} follows from \eqref{e:E5proof} and applying the operator $\frac{\sqrt{-1}}{2\pi p}\partial\ov{\partial}$ on the right-hand side of \eqref{e:5.16proof}.
\end{proof}

The almost sure convergence of $\frac{1}{p}[D_r(\bs_p)]$ as $p\to\infty$ was established in Theorem \ref{T:distrib}, together with a rate of convergence. 
We now present an alternative approach to proving almost sure convergence 
in this determinant framework, employing variance estimates in the spirit of 
the original argument of Shiffman–Zelditch for the line bundle case 
\cite[Section 3.3]{ShZ99} (see also \cite[Theorem 5.3.3]{MM07}).
%===
\begin{Theorem}[Variance estimate and almost sure convergence]\label{T:5.1Wick}
Let $\bs_p=(s^p_1,\ldots, s^p_r) \in H^0(X,L^p\otimes E)^r $ 
be given as in \eqref{e:random5}. Then there exists $C=C(r)>0$ 
such that for any smooth real test form 
$\varphi\in \Omega^{(n-1,n-1)}(X,\R)$ and all $p\gg 1$,
\begin{equation}
\mathrm{Var}[\langle [D_r(\bs_p)],\varphi\rangle]\leq C \left|\int_X \partial\ov{\partial}\varphi(x)\right|^2.
\label{e:VarE}
\end{equation}

As a consequence, we have 
\begin{equation}\label{eq:as-limit-k=r}
\frac1p\,[D_r(\bs_p)]
\;\longrightarrow\;
r\,c_1(L,h^L)
\end{equation}
in the weak convergence of $(1,1)$-currents on $X$ as $p\to\infty$, almost surely.
\end{Theorem}
\begin{proof}
Let $p$ be sufficiently large. Set the random variable for $x\in X$
\begin{equation}
    W_p(x)=\log \left(\frac{\|s^p_1\wedge\cdots\wedge s^p_r\|^2_{h^{pr}\otimes h^{\det E}}(x)}{\det P_p(x)}\right).
\end{equation}
Then, by Proposition \ref{P:WG-section} and the fact that for $\ell=1,\ldots, r$,
$$\E[|\log \chi^2(2\ell)|]<\infty\,, \, \E[|\log \chi^2(2\ell)|^2]<\infty,$$
we conclude that there exists a constant $c(r)>0$ such that for all $p\gg 1$, $x\in X$,
\begin{equation}
  \E[|W_p(x)|^2]\leq c(r).
  \label{e:5.20p}
\end{equation}

By \eqref{e:PL}, \eqref{e:E5proof}, and the Fubini-Tonelli theorem, we obtain that
\begin{equation}
    \mathrm{Var}[\langle [D_r(\bs_p)],\varphi\rangle]=\frac{1}{4\pi^2}\int_{(x,y)\in X\times X} (\partial\ov{\partial}\varphi(x))(\ov{\partial\ov{\partial}\varphi(y)}) \E[W_p(x)W_p(y)].
\end{equation}
Then, by the Cauchy-Schwarz inequality and \eqref{e:5.20p}, we have
\begin{equation}
    |\E[W_p(x)W_p(y)]|\leq \sqrt{\E[|W_p(x)|^2]\E[|W_p(y)|^2]}\leq c(r).
\end{equation}
This way, we obtain \eqref{e:VarE} with $C=\frac{c(r)}{4\pi^2}$.

Note that $\omega_L:=c_1(L,h^L)$ defines a K\"{a}hler metric on $X$. Using the topological meaning of $[D_r(\bs_p)]$, we obtain that there exists $C_0>0$ such that for any real test form $\varphi\in \Omega^{(n-1,n-1)}(X,\R)$ and for $p\gg 1$,
\begin{equation}
    \frac{1}{p}\left|\langle [D_r(\bs_p)],\varphi\rangle\right|\leq \frac{|\varphi|_{\mathscr{C}^0_L}}{p}\left|\langle [D_r(\bs_p)],\omega_L^{n-1}\rangle\right|\leq C_0|\varphi|_{\mathscr{C}^0_L}\int_X \omega_L^n,
    \label{e:uniform}
\end{equation}
where $|\varphi|_{\mathscr{C}^0_L}$ denotes the 
$\mathscr{C}^0(X)$-norm of $\varphi$ induced by $\omega_L$. 
Therefore, we can prove the weak convergence in \eqref{eq:as-limit-k=r} 
by considering a countable $\mathscr{C}^0$-dense family of test forms $\varphi$. 

By \eqref{e:VarE}, we have
$$\sum_{p\geq 1}  \mathrm{Var}\left[\frac{1}{p}\langle [D_r(\bs_p)],\varphi\rangle\right] <\infty.$$
Therefore, applying the Borel–Cantelli lemma, we conclude that
$$\frac{1}{p}\langle [D_r(\bs_p)],\varphi\rangle-\frac{1}{p}\langle \E[[D_r(\bs_p)]],\varphi\rangle \to 0\,,\, 
\text{ almost surely as } p\to\infty.$$
Finally, in combination with \eqref{e:randomE5} for 
$\frac{1}{p}\E[[D_r(\bs_p)]]$ and the argument using \eqref{e:uniform}, 
we obtain \eqref{eq:as-limit-k=r}.
\end{proof}

\end{document}